\documentclass{amsart}
\usepackage{amsmath, amssymb, bm, graphicx}
\usepackage{bbm}
\usepackage{psfrag}
\usepackage{mathrsfs}
\usepackage[usenames,dvipsnames]{color}
\usepackage{comment}
\usepackage{tikz}
\usepackage{epsfig,amsmath,amsfonts,amssymb,color,epic,rotating}
\usepackage[export]{adjustbox}
\theoremstyle{plain}             
\newtheorem{theorem}{Theorem}[section]
\newtheorem{lemma}[theorem]{Lemma}

\newtheorem{remark}[theorem]{Remark}

\definecolor{gray}{rgb}{0.9,0.9,0.9}
\definecolor{green}{rgb}{0.0,0.6,0.0}
\newcommand{\KL}{Karhunen-Lo\`eve\ }
\newcommand{\E}{\mathbb{E}}
\newcommand{\V}{\mathbb{V}}

\newcommand{\Cov}{\operatorname{Cov}}

\renewcommand{\div}{\operatorname{div}}

\newcommand{\Int}{\operatorname{Int}}

\renewcommand{\d}{\operatorname{d}\!}
\newcommand{\isdef}{\mathrel{\mathrel{\mathop:}=}}

\newcommand{\balpha}{{\boldsymbol{{\alpha}}}}

\newcommand{\bgamma}{{\boldsymbol{\gamma}}}

\newcommand{\brho}{{\boldsymbol{\rho}}}

\newcommand{\bs}[1]{{\boldsymbol{#1}}}
\makeatletter

\begin{document}
\title[Multilevel quadrature for elliptic parametric PDEs]{Multilevel quadrature for elliptic parametric 
partial differential equations in case of polygonal approximations
of curved domains}

\author{Michael Griebel}
\address{Michael Griebel,
Institut f\"ur Numerische Simulation, Universit\"at Bonn, Endenicher Allee 19b, 53115 Bonn, Deutschland und
Fraunhofer Institute for Algorithms and Scientific Computing (SCAI), Schloss Birlinghoven, 53754 Sankt Augustin, Deutschland}
\email{griebel@ins.uni-bonn.de}
\author{Helmut Harbrecht}
\address{Helmut Harbrecht,
Departement Mathematik und Informatik, Universit\"at Basel, Spiegelgasse 1, 4051 Basel, Schweiz\\
}
\email{helmut.harbrecht@unibas.ch}
\author{Michael D.\ Multerer}
\address{Michael D.\ Multerer,
Institute of Computational Science, Universit\`a della Svizzera italiana, 
Via Giuseppe Buffi 13, 6900 Lugano, Svizzera\\
}
\email{michael.multerer@usi.ch}
%
%
\maketitle
\begin{abstract}
Multilevel quadrature methods for parametric
operator equations such as the multilevel (quasi-) 
Monte Carlo method are closely related to the sparse 
tensor product approximation between the spatial
variable and the parameter. In this article,
we employ this fact and reverse the multilevel quadrature
method via the sparse grid construction by applying differences 
of quadrature rules to finite element discretizations of increasing 
resolution. Besides being algorithmically more efficient if the 
underlying quadrature rules are nested, this way of performing 
the sparse tensor product approximation enables the easy use 
of non-nested and even adaptively refined finite element meshes. 
Especially, we present a rigorous error and regularity analysis 
of the fully discrete
solution, taking into account the effect of polygonal approximations
to a curved physical domain and the numerical approximation of
the bilinear form. Our results facilitate the construction of
efficient multilevel quadrature methods based on deterministic
quadrature rules. Numerical results in three
spatial dimensions are provided to illustrate the approach.
\end{abstract}

\section{Introduction}
The present article is concerned with the numerical solution 
of elliptic parametric second order boundary value problems of the form
\begin{equation}\label{eq:modprob1}
-\div\big(a({\bs y})\nabla u({\bs y})\big)=f({\bs y})\text{ in }D,
\quad u({\bs y})=0\text{ on }\partial D,\quad {\bs y}\in{{\Gamma}},
\end{equation}
where \(D\subset\mathbb{R}^d\) denotes the spatial domain 
and \({{\Gamma}}\subset\mathbb{R}^m\) denotes the parameter domain.
Prominent representatives of such problems arise from recasting 
boundary value problems with random data, like random diffusion 
coefficients, random right hand sides and even random domains.
A high-dimensional parametric boundary value problem of the 
form \eqref{eq:modprob1} is then derived by inserting the 
truncated \KL expansion of the random data, 
see e.g.\  \cite{BNT,BTZ,FST,HPS14b,MK}. 
Hence, the computation of quantities of interest 
amounts to a high-dimensional 
Bochner integration problem. The latter can be dealt with by quadrature 
methods. Since every quadrature method requires the repeated evaluation 
of the integrand for different sample or quadrature points,
we have to compute the solution to \eqref{eq:modprob1} with respect to
many different values of the parameter \({\bs y}\in{{\Gamma}}\).

An efficient approach to deal with the quadrature problem is 
the multilevel Monte Carlo method (MLMC), which has been 
developed in \cite{BSZ,G,GW,H,H2}. As first observed in 
\cite{GH,HPS1}, this approach mimics a certain sparse grid 
approximation between the physical space and the parameter
space. Thus, the extension to the multilevel quasi-Monte Carlo (MLQMC)
method and even more general multilevel quadrature methods is
obvious. In this article, we focus on such deterministic quadrature 
methods, which, in particular, require extra regularity of 
the solution in terms of spaces of dominant mixed derivatives, 
cf.\ \cite{HPS1,HPS2,KSS3} for example. This extra regularity is 
available for important classes of parametric problems, see
{\cite{CDS,CDS11}} for the case of affine elliptic diffusion coefficients 
and \cite{HS11} for the case of log-normally distributed diffusion 
coefficients. For the sake of clarity in presentation,
we shall focus here on affine elliptic diffusion problems as they occur from 
the discretization of uniformly elliptic random diffusion coefficients.
We put our emphasis on the rigorous error and regularity analysis of 
the fully discrete solution, taking into account the effect of polygonal 
approximations to a curved physical domain and the numerical 
approximation of the bilinear form.

In addition, we focus on a particular construction of the
multilevel quadrature, which is very well suited for the 
use with black-box finite element solvers and adaptive
mesh refinements:
Taking the fact that a multilevel quadrature scheme
resembles a sparse tensor product approximation 
between the spatial variable
and the parametric variable as a starting point, we
have the following abstract framework. Let
\[
  V_0^{(i)} \subset V_1^{(i)}\subset\cdots\subset V_j^{(i)}
  	\subset\cdots\subset\mathcal{H}_i,\quad i = 1,2,
\]
denote two sequences of finite dimensional sub-spaces with 
increasing approximation power in some linear spaces 
$\mathcal{H}_i$. To approximate a given object of the tensor
product space $\mathcal{H}_1\otimes\mathcal{H}_2$, one 
canonically considers the {full tensor product spaces} 
$U_j\isdef V_j^{(1)}\otimes V_j^{(2)}$. However, the cost $\dim 
U_j = \dim V_j^{(1)}\cdot \dim V_j^{(2)}$ is often too 
expensive. To reduce this cost, one might consider
the approximation in so-called {\em sparse grid spaces}, 
see e.g.\ \cite{BG}. For $\ell\ge 0$, one introduces the 
complement spaces
\[
  W_{\ell+1}^{(i)} = V_{\ell+1}^{(i)}\ominus V_\ell^{(i)},
  	\quad i=1,2,
\]
which gives rise to the multilevel decompositions
\begin{equation}	\label{eq:ml}
  V_j^{(i)} = \bigoplus_{\ell=0}^j W_\ell^{(i)}, \quad W_0^{(i)}:=V_0^{(i)},  	\quad i=1,2.
\end{equation}
Then, the sparse grid space is defined by
\begin{equation}	\label{eq:sg}
  \widehat{U}_j := \bigoplus_{\ell+\ell'\le j} W_\ell^{(1)}\otimes W_{\ell'}^{(2)}.
\end{equation}
Under the assumptions that the dimensions of 
$\big\{V_\ell^{(1)}\big\}$ and $\big\{V_\ell^{(2)}\big\}$ 
form geometric series, \eqref{eq:sg} contains, at most up 
to a logarithm, only $\mathcal{O}\big(\max\big\{\dim V_j^{(1)}, 
\dim V_j^{(2)}\big\}\big)$ degrees of freedom. Nevertheless, it 
offers nearly the same approximation power as $U_j$ provided 
that the object to be approximated has some extra smoothness 
by means of mixed regularity. For further details, see \cite{GH1}.

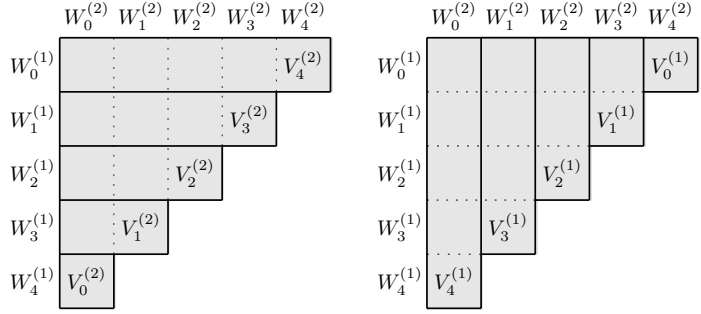
\begin{figure}
\begin{center}
\scalebox{0.8}{
\setlength{\unitlength}{0.9cm}
\begin{picture}(6,6)
  \put(1,4.1){\colorbox{gray}{\makebox(4.8,0.8){}}}
  \put(1,3.1){\colorbox{gray}{\makebox(3.8,0.8){}}}
  \put(1,2.1){\colorbox{gray}{\makebox(2.8,0.8){}}}
  \put(1,1.1){\colorbox{gray}{\makebox(1.8,0.8){}}}
  \put(1,0.1){\colorbox{gray}{\makebox(0.8,0.8){}}}
\thicklines
  \put(1,5){\line(1,0){5}}
  \put(1,4){\line(1,0){5}}
  \put(6,4){\line(0,1){1}}
  \put(1,3){\line(1,0){4}}
  \put(5,3){\line(0,1){1}}
  \put(1,2){\line(1,0){3}}
  \put(4,2){\line(0,1){1}}
  \put(1,1){\line(1,0){2}}
  \put(3,1){\line(0,1){1}}
  \put(1,0){\line(1,0){1}}
  \put(2,0){\line(0,1){1}}
  \put(1,0){\line(0,1){5}}
\thinlines
  \dashline{0.05}(2,1)(2,5)
  \dashline{0.05}(3,2)(3,5)
  \dashline{0.05}(4,3)(4,5)
  \dashline{0.05}(5,4)(5,5)

  \put(0.5,4.5){\makebox(0,0){${W}_0^{(1)}$}}
  \put(0.5,3.5){\makebox(0,0){${W}_1^{(1)}$}}
  \put(0.5,2.5){\makebox(0,0){${W}_2^{(1)}$}}
  \put(0.5,1.5){\makebox(0,0){${W}_3^{(1)}$}}
  \put(0.5,0.5){\makebox(0,0){${W}_4^{(1)}$}}
  \put(5.5,4.5){\makebox(0,0){${V}_4^{(2)}$}}
  \put(4.5,3.5){\makebox(0,0){${V}_3^{(2)}$}}
  \put(3.5,2.5){\makebox(0,0){${V}_2^{(2)}$}}
  \put(2.5,1.5){\makebox(0,0){${V}_1^{(2)}$}}
  \put(1.5,0.5){\makebox(0,0){${V}_0^{(2)}$}}
  \put(5.5,5.4){\makebox(0,0){${W}_4^{(2)}$}}
  \put(4.5,5.4){\makebox(0,0){${W}_3^{(2)}$}}
  \put(3.5,5.4){\makebox(0,0){${W}_2^{(2)}$}}
  \put(2.5,5.4){\makebox(0,0){${W}_1^{(2)}$}}
  \put(1.5,5.4){\makebox(0,0){${W}_0^{(2)}$}}
\end{picture}\qquad
\begin{picture}(6,6)
  \put(1,0.1){\colorbox{gray}{\makebox(0.8,4.8){}}}
  \put(2,1.1){\colorbox{gray}{\makebox(0.8,3.8){}}}
  \put(3,2.1){\colorbox{gray}{\makebox(0.8,2.8){}}}
  \put(4,3.1){\colorbox{gray}{\makebox(0.8,1.8){}}}
  \put(5,4.1){\colorbox{gray}{\makebox(0.8,0.8){}}}
\thicklines
  \put(1,5){\line(1,0){5}}
  \put(5,4){\line(1,0){1}}
  \put(6,4){\line(0,1){1}}
  \put(4,3){\line(1,0){1}}
  \put(5,3){\line(0,1){2}}
  \put(3,2){\line(1,0){1}}
  \put(4,2){\line(0,1){3}}
  \put(2,1){\line(1,0){1}}
  \put(3,1){\line(0,1){4}}
  \put(1,0){\line(1,0){1}}
  \put(2,0){\line(0,1){5}}
  \put(1,0){\line(0,1){5}}
\thinlines
  \dashline{0.05}(1,1)(2,1)
  \dashline{0.05}(1,2)(3,2)
  \dashline{0.05}(1,3)(4,3)
  \dashline{0.05}(1,4)(5,4)
  
  \put(0.5,4.5){\makebox(0,0){${W}_0^{(1)}$}}
  \put(0.5,3.5){\makebox(0,0){${W}_1^{(1)}$}}
  \put(0.5,2.5){\makebox(0,0){${W}_2^{(1)}$}}
  \put(0.5,1.5){\makebox(0,0){${W}_3^{(1)}$}}
  \put(0.5,0.5){\makebox(0,0){${W}_4^{(1)}$}}
  \put(5.5,4.5){\makebox(0,0){${V}_0^{(1)}$}}
  \put(4.5,3.5){\makebox(0,0){${V}_1^{(1)}$}}
  \put(3.5,2.5){\makebox(0,0){${V}_2^{(1)}$}}
  \put(2.5,1.5){\makebox(0,0){${V}_3^{(1)}$}}
  \put(1.5,0.5){\makebox(0,0){${V}_4^{(1)}$}}
  \put(5.5,5.4){\makebox(0,0){${W}_4^{(2)}$}}
  \put(4.5,5.4){\makebox(0,0){${W}_3^{(2)}$}}
  \put(3.5,5.4){\makebox(0,0){${W}_2^{(2)}$}}
  \put(2.5,5.4){\makebox(0,0){${W}_1^{(2)}$}}
  \put(1.5,5.4){\makebox(0,0){${W}_0^{(2)}$}}
\end{picture}}
\caption{\label{fig:sg}Different representations of
the sparse grid space.}
\end{center}
\end{figure}

In view of \eqref{eq:ml}, factoring out with respect to the 
first component, one can rewrite \eqref{eq:sg} according to
\begin{equation}\label{eq:repr1}
  \widehat{U}_j = \bigoplus_{\ell=0}^j W_\ell^{(1)}\otimes 
  	\bigg(\bigoplus_{\ell'=0}^{j-\ell} W_{\ell'}^{(2)}\bigg)
  		= \bigoplus_{\ell=0}^j W_\ell^{(1)}\otimes V_{j-\ell}^{(2)}.
\end{equation}
{This representation has already been proposed in \cite{GH1}.}
Obviously, in complete analogy there holds
\begin{equation}\label{eq:repr2}
  \widehat{U}_j = \bigoplus_{\ell'=0}^j \bigg(\bigoplus_{\ell=0}^{j-\ell'} 
  	W_{\ell}^{(1)}\bigg)\otimes W_{\ell'}^{(2)}
  		= \bigoplus_{\ell=0}^j V_{j-\ell}^{(1)}\otimes W_\ell^{(2)}.
\end{equation}
We refer to Fig.~\ref{fig:sg} for an illustration, where the left plot
corresponds to the representation \eqref{eq:repr1} and the right
plot corresponds to the representation \eqref{eq:repr2}. The advantage
of the representation \eqref{eq:repr1} is that we can give up 
the requirement that the spaces $\{V_\ell^{(2)}\}$ are nested.
Likewise, for the representation \eqref{eq:repr2}, the spaces 
$\{V_\ell^{(1)}\}$ need not to be nested any more.

In the context of the parametric diffusion problem \eqref{eq:modprob1}, 
we aim at computing
\[
\int_{{\Gamma}} \mathcal{F}\big(u({\bs y})\big)\rho({\bs y})\d{\bs y},
\]
where \(\rho\) is the density of some measure on
\(\Gamma\) and \(\mathcal{F}\)
denotes a functional or, as in the case of moment
computation, it may be defined as \(\mathcal{F}\big(u({\bs y})\big)
= u^p({\bs y})\) for \(p=1,2,\ldots\).
In this context, $\{V_\ell^{(1)}\}$ corresponds to a sequence of finite element 
spaces and $\{V_\ell^{(2)}\}$ refers to a sequence of quadrature rules. 
If we denote the finite element solutions of \eqref{eq:modprob1} by 
$\mathfrak{u}_\ell({\bs y})\in V_\ell^{(1)}$ and if we denote the sequence 
of quadrature rules by \(Q_{\ell'}\colon C({{\Gamma}})\to\mathbb{R}\), we 
thus arrive with respect to \eqref{eq:repr1} at the decomposition
\begin{equation}\label{eq:nestedV}
\int_{{\Gamma}} \mathcal{F}\big(u({\bs y})\big)\rho({\bs y})\d{\bs y}\approx
\sum_{\ell=0}^j Q_{j-\ell}\Delta\mathcal{F}_\ell\big(u({\bs y})\big),
\end{equation}
where 
$\Delta\mathcal{F}_\ell\big(u({\bs y})\big)\isdef \mathcal{F}
\big(\mathfrak{u}_\ell({\bs y})\big)-\mathcal{F}\big(\mathfrak{u}_{\ell-1}
({\bs y})\big)$ and $\mathcal{F}\big(\mathfrak{u}_{-1}({\bs y})\big)\isdef 0$,
{see \cite{HPS1}}.
On the other hand, similarly to \eqref{eq:repr2}, we obtain
the decomposition
\begin{equation}\label{eq:nestedQ}
\int_{{\Gamma}} \mathcal{F}\big(u({\bs y})\big)\rho({\bs y})\d{\bs y}\approx
\sum_{\ell=0}^j \Delta Q_{\ell}\mathcal{F}\big(\mathfrak{u}_{j-\ell}({\bs y})\big),
\end{equation}
where $\Delta Q_{\ell}\isdef Q_{\ell}-Q_{\ell-1}$ and $Q_{-1}\isdef 0$.
Both representations are equivalent but have a 
different impact on its numerical implementation.

Originally, {multilevel quadrature methods} have been interpreted 
as variance reduction methods for the Monte Carlo quadrature, a 
view which has originally been introduced for the approximation of parametric integrals, 
cf.~\cite{H,H2}. Consequently, the representation \eqref{eq:repr1}, 
and thus the decomposition \eqref{eq:nestedV}, has been used in 
previous articles, see, for example, \cite{G,Gil15} for stochastic 
ordinary differential equations and {\cite{BSZ,HPS1,TSU13,TJWG15}} for 
partial differential equations with random data. To this end, usually a 
nested sequence of approximation spaces is presumed such 
that the complement spaces $\big\{W_\ell^{(1)}\big\}$ are well-defined.
In the context of partial differential equations, these complement spaces 
are given via the difference of Galerkin projections onto subsequent 
finite element spaces. This circumstance can be avoided in the case 
of \(\mathcal{F}\) being a functional, cf.~{\cite{HANvST15,TSU13}}. 
Still, we emphasize that, particularly in the context
of the Monte Carlo method, there are already results available, which allow
for giving up this nestedness, see e.g.~\cite{CST13,TSU13}. A more general 
result addressing the resulting error in the underlying bilinear form can 
be found in \cite{Sieb15}.

The decomposition \eqref{eq:nestedV} is well suited if the spatial 
dimension is small, as it is the case for one-dimensional partial 
differential equations with random data or for stochastic ordinary 
differential equations. Nevertheless, in two or three spatial dimensions, 
the construction of nested approximation spaces might  be difficult
or even not be possible at all. Sometimes, in view of adaptive 
refinement strategies, it might be favourable to give up
nestedness. In the article at hand, we employ the 
decomposition \eqref{eq:nestedQ}. It allows more naturally
for non-nested finite element spaces which, in turn, induce 
different approximations of the underlying domain.
Moreover, using nested quadrature 
formulae, a considerable speed-up is achieved in comparison to 
the conventional multilevel quadrature which is based on the 
representation \eqref{eq:nestedV}.

The rest of the article is organized as follows. We start by
introducing the underlying random model in Section \ref{sec:reformulation}
and perform the parametric reformulation that results in
\eqref{eq:modprob1}.
Then, the next two sections are 
dedicated to the discretization, namely the quadrature 
rule for the parametric variable (Section~\ref{sec:QM})
and the finite element discretization for the physical 
domain (Section~\ref{sec:FEM}). The multilevel 
quadrature for the model problem is discussed in Section
\ref{sec:MLQ}. 
In Section \ref{sec:error}, we present the error and regularity analysis for
the multilevel quadrature taking into account polygonal approximations
of curved domains. We emphasize that the key result in this section,
namely Lemma~\ref{lem:decay}, is robust with respect to the parameter dimension \(m\).
Afterwards, in Section~\ref{sec:consistency}, we consider
a fully discrete approximation of the solution to \eqref{eq:modprob1}
and take also quadrature errors in the bilinear form into account.
Again, the main result (Theorem~\ref{thm:fullydisc}) is robust 
with respect to the parameter dimension \(m\). Finally, 
in Section \ref{sec:results}, we provide numerical results in
three spatial dimensions to validate our approach.

Throughout this article, in order to avoid the repeated use of generic but
unspecified constants, we mean by $C \lesssim D$ that $C$ can be
bounded by a multiple of $D$, independently of parameters
which $C$ and $D$ may depend on. Obviously, $C \gtrsim D$
is defined as $D \lesssim C$, and $C \sim D$ as $C \lesssim D$
and $C \gtrsim D$.

\section{Problem setting}\label{sec:reformulation}
Let $(\Omega,\Sigma,\mathbb{P})$ be a complete and separable 
probability space with $\sigma$-field $\Sigma\subset 2^\Omega$ 
and probability measure $\mathbb{P}$. 
We intend to compute the expectation
\[
  \mathbb{E}[u] = \int_\Omega u(\omega)\d\mathbb{P}(\omega)
  	\in H_0^1(D)
\]
and the variance 
\[
  \mathbb{V}[u] = \int_\Omega \big\{u(\omega)
	-\mathbb{E}[u]\big\}^2\d\mathbb{P}(\omega)
		\in W_0^{1,1}(D)
\]
of the random function $u(\omega)\in H_0^1(D)$ which 
solves the stochastic diffusion problem
\begin{equation}\label{eq:modprob}
  -\div\big({{a}}(\omega)\nabla u(\omega)\big) =f\ \text{in}\ D
  \ \text{for almost every $\omega\in\Omega$}.
\end{equation} 

For sake of simplicity, we assume that the stochastic 
diffusion coefficient is given by a finite \KL expansion
\begin{equation}\label{KLa}
  {{a}}({\bs x},\omega)=\mathbb{E}[{{a}}]({\bs x})
 	+\sum_{k=1}^m\sqrt{\lambda_k}\varphi_k({\bs x})\psi_k(\omega)
\end{equation}
with pairwise \(L^2\)-orthonormal functions $\varphi_k\in L^\infty(D)$ and
stochastically independent random variables $\psi_k(\omega)\in [-1,1]$. 
Especially, it is assumed that the random variables admit continuous 
density functions $\rho_k\colon [-1,1] \to\mathbb{R}$ with respect to 
the Lebesgue measure.

In practice, one generally has to compute the expansion \eqref{KLa}
from the given covariance kernel 
\[
  \Cov[{{a}}]({\bs x},{\bs x}') 
  	= \int_\Omega \big\{{{a}}({\bs x},\omega)-\E[{{a}}]({\bs x})\big\}
		\big\{{{a}}({\bs x}',\omega)-\E[{{a}}]({\bs x}')\big\} \d\mathbb{P}(\omega).
\]
If the expansion contains infinitely many terms, it has to be appropriately 
truncated which will induce an additional discretization error. For 
details, we refer the reader to \cite{GS,HPS14a,L,ST2}.

The assumption that the random variables $\{\psi_k(\omega)\}$ 
are independent implies that the joint 
density function of the random variables 
is given by
\(
{\rho} ({\bs y}) \isdef\prod_{k=1}^m \rho_k(y_k).
\)

Thus, we are able to reformulate the stochastic problem \eqref{eq:modprob} 
as a parametric, deterministic problem in \(L^2_{{\rho}}({{\Gamma}})\). To this end, the 
probability space \(\Omega\) is identified with its image \({{\Gamma}}\isdef[-1,1]^m\) with respect
to the measurable mapping
\[
\boldsymbol{\psi}\colon\Omega\to{{\Gamma}},\quad \omega\mapsto \boldsymbol{\psi}(\omega)
	\isdef\big(\psi_1(\omega),\ldots,\psi_m(\omega)\big).
\]
Hence, the random variables \(\psi_k\) are substituted by coordinates \(y_k\in[-1,1]\). 

We introduce  the measure \({\rho}(\bs{y})\d{\bs y} \) on \({{\Gamma}}\),
which is defined by the product density function
\(
{\rho} ({\bs y})\isdef \prod_{k=1}^m \rho_k(y_k).
\)

Next, in order to ensure \(H^2\)-regularity of the model 
problem, let \(D\subset\mathbb{R}^d\), \(d=2,3\), be either a convex, polygonal 
domain. We consider the parametric diffusion problem
\begin{equation}	\label{eq:parprob}
\begin{aligned}
&\text{find \(u\in{L^2_{{\rho}}\big({{\Gamma}}; H_0^1(D)\big)}\) such that}\\
&\qquad-\div\big({{a}}({\bs y})\nabla u({\bs y})\big)=f\ \text{in $D$}\ 
	\text{for almost every \({\bs y} \in {{\Gamma}}\)},
\end{aligned}
\end{equation}
with \(f\in L^2(D)\) and \({{a}}\colon D\times{{\Gamma}}\to\mathbb{R}\) with
\begin{equation}\label{eq:paramalpha}
  {{a}}({\bs x},{\bs y}) = \varphi_0({\bs x})
  	+\sum_{k=1}^m\sqrt{\lambda_k}\varphi_k({\bs x}) y_k,\quad\gamma_k\isdef\sqrt{\lambda_k}\|\varphi_k\|_{W^{1,\infty}(D)}.
\end{equation}
Note that {\(u\in L^2_{{\rho}}\big({{\Gamma}}; H_0^1(D)\big)\)
guarantees finite second order moments of the solution.}

By the Lax-Milgram theorem, the unique solvability of
the parametric diffusion problem \eqref{eq:parprob}
in $L_{\rho}^2\big({{\Gamma}};H_0^1(D)\big)$ follows immediately 
if we impose the condition
\begin{equation}\label{eq:alpha}
  0<{{a}}_{\min}\le{{a}}({\bs y})\le{{a}}_{\max}<\infty
  	\ \text{in $D$}
\end{equation}
for all ${\bs y}\in{{\Gamma}}$ on the diffusion coefficient. 
Moreover, we obtain the stability estimate
\[
\|u({\bs y})\|_{H^1(D)}\leq\frac{1}{a_{\min}}\|f\|_{H^{-1}(D)}
\lesssim\frac{1}{a_{\min}}\|f\|_{L^2(D)}\quad\text{for almost every }{\bs y}\in{{\Gamma}}.
\]
Hence, the solution to \eqref{eq:parprob} is essentially bounded
with respect to \({\bs y}\in{{\Gamma}}\).


In {e.g.~\cite{BNTT12,CDS,CDS11,ES14,ST3}}, it has been proven that the solution $u$ of 
\eqref{eq:parprob} is analytical as mapping $u\colon{{\Gamma}}\to 
H_0^1(D)$. Moreover, it has been shown in \cite{CDS} that \(u\)
is even an analytical mapping $u\colon{{\Gamma}}\to\mathcal{W}\isdef 
H_0^1(D)\cap H^2(D)$ given that the \(\{\varphi_k\}\) in \eqref{eq:paramalpha}
belong to $W^{1,\infty}(D)$. This constitutes the
necessary mixed regularity for a sparse tensor product discretization, see e.g.
{\cite{HPS2,KSS3}}.
A similar result for diffusion problems
with coefficients of the form \(\exp\big({{a}}({\bs x},{\bs y})\big)\) has been
shown in \cite{HS11}.

Since \(u\) is supposed to be in \(L^2_{{\rho}}\big({{\Gamma}}; H_0^1(D)\big)\),
we can compute its expectation
\begin{equation}	\label{eq:E}
\E[u]=\int_{{{\Gamma}}} 
	u({\bs y}){\rho}({\bs y})\d{\bs y}\in H_0^1(D)
\end{equation}
and its variance 
\begin{equation}	\label{eq:V}
\V[u] = \E[u^2]-\E[u]^2
	= \int_{{\Gamma}} u^2({\bs y}){\rho}({\bs y})\d{\bs y}
		-\mathbb{E}[u]^2\in W_0^{1,1}(D).
\end{equation}
We will focus in the sequel on the efficient numerical 
computation of these possibly high-dimensional integrals.


\section{Quadrature in the parameter space}\label{sec:QM}
The expectation and the variance of the solution \(u\) to \eqref{eq:parprob}
are given by the integrals 
\eqref{eq:E} and \eqref{eq:V}. To compute these integrals, 
we employ a sequence of quadrature formulae 
$\{Q_\ell\}$ for the Bochner integral
\[
  \Int\colon L_{\rho}^1({{\Gamma}};\mathcal{X})\to \mathcal{X},\quad  
  \Int v = \int_{{\Gamma}} v({\bs y})\rho({\bs y})\d{\bs y}
\]
where $\mathcal{X}\subset L^2(D)$ denotes some Banach space.
The quadrature formula
\begin{equation}	\label{eq:quadrature}
  Q_\ell\colon L_{\rho}^1({{\Gamma}};\mathcal{X})\to\mathcal{X},\quad  
  (Q_\ell v)({\bs x}) = \sum_{i=1}^{N_\ell}\omega_{\ell,i} 
  	v({\bs x},\boldsymbol\xi_{\ell,i})\rho(\boldsymbol\xi_{\ell,i})
\end{equation}
is supposed to provide the error bound
\begin{equation}	\label{eq:Q1}
  \|(\Int-Q_\ell)v\|_\mathcal{X}\lesssim\varepsilon_{\ell}\|v\|_{\mathcal{H}({{\Gamma}};\mathcal{X})}
\end{equation}
uniformly in $\ell\in\mathbb{N}$, where $\mathcal{H}({{\Gamma}};\mathcal{X})
\subset L_{\rho}^2({{\Gamma}},\mathcal{X})$ is a suitable Bochner 
space. 
{Note that since the density \(\rho\) is fixed, it will be suppressed in the upcoming error estimates and will, thus, be hidden in the constants.}

The following particular examples of quadrature rules 
\eqref{eq:quadrature} are considered in our numerical 
experiments:
\begin{itemize}
  \item[$\bullet$]
  The Monte Carlo method satisfies \eqref{eq:Q1} only
  with respect to the root mean square error. Namely, it holds
\[ 
  \sqrt{\E\big(\|(\Int-Q_\ell)v\|_\mathcal{X}^2\big)}\lesssim
  	\varepsilon_{\ell}\|v\|_{\mathcal{H}({{\Gamma}};\mathcal{X})}
\]  
  with $\varepsilon_\ell=N_\ell^{-1/2}$ and $\mathcal{H}({{\Gamma}};
  \mathcal{X})=L_{\rho}^2({{\Gamma}};\mathcal{X})$.   
  \item[$\bullet$]
  The quasi-Monte Carlo method leads typically 
  to $\varepsilon_\ell= N_\ell^{-1}(\log N_\ell)^{m}$, {where it is sufficient to consider the}
  Bochner space $\mathcal{H}({{\Gamma}};\mathcal{X})=W_{\text{mix}}^{1,1}
  ({{\Gamma}};\mathcal{X})$ of all equivalence classes of 
 functions $v\colon{{\Gamma}}\to\mathcal{X}$ with finite norm
\begin{equation}	\label{eq:W11mix}
  \|v\|_{W_{\text{mix}}^{1,1}({{\Gamma}};\mathcal{X})} \isdef
  	\sum_{\|{\bs q}\|_\infty\le 1}\int_{{\Gamma}}\bigg\|\frac{\partial^{\|{\bs q}\|_1}}
		{\partial y_1^{q_1}\partial y_2^{q_2}\cdots\partial y_m^{q_m}} 
			v({\bs y})\bigg\|_{\mathcal{X}}\d {\bs y}<\infty,
\end{equation}
  see e.g.~\cite{Niederreiter}. Note that, {in this case}, the estimate requires
  that the densities satisfy $\rho_k\in W^{1,\infty}(-1,1)$. For the Halton sequence, cf.~\cite{Hal60},
  it can even be shown that $\varepsilon_\ell= N_\ell^{\delta-1}$ for arbitrary \(\delta>0\)
  given that the spatial functions in \eqref{eq:paramalpha} satisfy 
  \(\gamma_k\lesssim k^{-3-\varepsilon}\) for arbitrary \(\varepsilon>0\).
  {This is a straightforward consequence from the results in \cite{Wan02}, see e.g.~\cite{HPS14b}.}
  \item[$\bullet$]
  Let the densities $\rho_k$ be in $W^{r,\infty}(-1,1)$.
  If $v\colon{{\Gamma}}\to\mathcal{X}$ has mixed regularity 
  of order $r$ with respect to the parameter ${\bs y}$, i.e.
  \begin{equation}\label{eq:Crnorm}
  \|v\|_{W_{\text{mix}}^{r,\infty}({{\Gamma}};\mathcal{X})}
  \isdef \max_{\|\balpha\|_\infty\leq r}\big\|\partial^\balpha_{\bs y}v\big\|_{L^\infty({{\Gamma}};\mathcal{X})}<\infty,
  \end{equation}
  then one can apply a (sparse) tensor product Clenshaw-Curtis 
  quadrature rule. This yields the convergence rate $\varepsilon_\ell
  = 2^{-\ell r}\ell^{m-1}\), where \(N_\ell\sim 2^\ell\ell^{m-1}\) and 
  $\mathcal{H}({{\Gamma}};\mathcal{X})=W_{\text{mix}}^{r,\infty}
  ({{\Gamma}};\mathcal{X})$, see \cite{NR96}.\footnote{The Clenshaw-Curtis quadrature 
  converges exponentially if the integrand $v\colon{{\Gamma}}\to\mathcal{X}$ 
  and the density $\rho$ are analytic.} 
\end{itemize}

For our purposes, we shall assume that the number $N_\ell$
of points of the quadrature formula ${Q}_\ell$ is chosen such that the 
corresponding accuracy is
\begin{equation}	\label{eq:Q2}
  \varepsilon_\ell= 2^{-\ell}.
\end{equation}
Then, for the respective difference quadrature $\Delta Q_\ell
\isdef Q_\ell-Q_{\ell-1}$, we immediately obtain by combining
\eqref{eq:Q1} and \eqref{eq:Q2} the error bound
\begin{align*}
\|\Delta Q_\ell v\|_{\mathcal{X}} =
\|(Q_\ell-Q_{\ell-1}) v\|_{\mathcal{X}}
&\leq \|(\Int-Q_\ell) v\|_{\mathcal{X}}+\|(\Int-Q_{\ell-1}) v\|_{\mathcal{X}}\\
&\lesssim 2^{-\ell}\|v\|_{\mathcal{H}({{\Gamma}};\mathcal{X})}.
\end{align*}

\section{Finite element approximation in the spatial variable}\label{sec:FEM}
In order to apply the quadrature formula \eqref{eq:quadrature}, we 
shall calculate the solution $u({\bs y})\in H_0^1(D)$ of the diffusion 
problem \eqref{eq:parprob} in certain points ${\bs y}\in{{\Gamma}}$. To 
this end, consider a not necessarily nested sequence of shape regular 
and quasi-uniform triangulations or tetrahedralizations $\{\mathcal{T}_\ell\}$ 
for \(\ell\geq 0\) of the domain $D$, respectively, 
each of which with the mesh size $h_\ell\sim 2^{-\ell}$. 
If the domain is not polygonal, then we obtain a polygonal approximation 
$D_\ell$ of the domain $D$ by replacing curved edges and faces by planar ones. 

In order to deal only with the fixed domain $D$ and not with 
the different polygonal approximations $D_\ell$, we follow 
\cite{B} and extend functions defined on $D_\ell$ by zero onto 
$D\setminus D_\ell$. Hence, given the triangulation or the 
tetrahedralization $\{\mathcal{T}_\ell\}$, we define the spaces
\begin{align*}
  \mathcal{S}_\ell(D)\isdef\{v\in C(D):v|_T\ & \text{is a linear polynomial for all}\ 
  T\in\mathcal{T}_\ell\\ &\qquad\qquad\text{and}\ v({\bs x})=0\ 
	\text{for all nodes ${\bs x}\in \partial D$}\}
\end{align*}
of continuous, piecewise linear finite elements.
Notice that it does hold $\mathcal{S}_\ell(D)\subset H^1(D)$
but in general $\mathcal{S}_\ell(D)\not\subset H_0^1(D)$.

We shall further introduce the finite element solution 
$\mathfrak{u}_\ell({\bs y})\in\mathcal{S}_\ell(D)$ of 
\eqref{eq:parprob} which satisfies
\begin{equation}\label{eq:blf}
\mathcal{B}_{\bs y}(\mathfrak{u}_\ell,v_\ell)
  \isdef\int_D a({\bs x},{\bs y})\nabla \mathfrak{u}_\ell({\bs x},{\bs y})\nabla v_\ell({\bs x})\d{\bs x}
  = \int_D f({\bs x})v_\ell({\bs x})\d{\bs x}
\end{equation}
for all $v_\ell\in\mathcal{S}_\ell(D)$.
If \(D\neq D_\ell\), the bilinear form $\mathcal{B}_{\bs y}(\cdot,\cdot)$ 
is also well defined for functions from $\mathcal{S}_\ell(D)$ since 
$\mathcal{S}_\ell(D)\subset H^1(D)$. Nevertheless, in order to maintain 
the ellipticity of the bilinear form, we shall assume that the 
mesh size \(h_0\) is sufficiently small to ensure that functions 
in $\mathcal{S}_\ell(D)$ are zero on a part of the boundary of $D$.

It is shown in e.g.~\cite{B,Brenner} that the finite element solution 
$\mathfrak{u}_\ell({\bs y})\in\mathcal{S}_\ell(D)$ of \eqref{eq:blf}
admits the following approximation properties.

\begin{lemma}\label{lem:femerror}
Consider a convex, polygonal domain $D$ or a domain 
with $C^2$-smooth boundary and let $f\in L^2(D)$. Then, 
the finite element solution $\mathfrak{u}_\ell({\bs y})\in \mathcal{S}_\ell(D)$ 
of the diffusion problem \eqref{eq:parprob} and respectively its square 
$\mathfrak{u}_\ell^2({\bs y})$
satisfy the error estimate
\begin{equation}	\label{eq:err-est-1}
 \big\|u^p({\bs y})-\mathfrak{u}_\ell^p({\bs y})\big\|_{\mathcal{X}}\lesssim h_\ell\|f\|_{L^2(D)}^p,
\end{equation}
where \(\mathcal{X}=H^1(D)\) for \(p=1\) and   \(\mathcal{X}=W^{1,1}(D)\) for \(p=2\).
The constants hidden in \eqref{eq:err-est-1}
depend on ${{a}}_{\min}$ and ${{a}}_{\max}$, but not on 
${\bs y}\in{{\Gamma}}$. 
\end{lemma}

\section{The multilevel quadrature method}\label{sec:MLQ}
Based on the nomenclature from the previous sections, we 
now introduce the multilevel quadrature in a formal way.
To that end, let \(u\in\mathcal{H}({{\Gamma}};H^2(D))\), where the 
underlying Bochner space is determined by the quadrature under
consideration. For the sequence \(\{\mathfrak{u}_\ell({\bs y})\}_{\ell}\) 
of finite element solutions, there obviously holds 
\(
\lim_{\ell\to\infty} \mathfrak{u}_\ell({\bs y}) = u({\bs y})
\)
uniformly in \({\bs y}\in{{\Gamma}}\). Thus, if \(\mathcal{F}\) 
is continuous, we obtain
\begin{equation}\label{eq:prop1}
\lim_{\ell\to\infty}\mathcal{F}\big(\mathfrak{u}_\ell({\bs y})\big)=\mathcal{F}\big(u({\bs y})\big)
\end{equation}
also uniformly in \({\bs y}\in{{\Gamma}}\). Moreover, we have 
for the sequence \(\{Q_\ell\}_\ell\) of quadrature rules and
for a sufficiently smooth integrand that
\begin{equation}\label{eq:prop2}
\lim_{\ell\to\infty} Q_\ell v=\int_{{\Gamma}} v({\bs y})\rho({\bs y})\d{\bs y}.
\end{equation}
The combination of the relations \eqref{eq:prop1} and 
\eqref{eq:prop2} leads to
\[
\int_{{\Gamma}} \mathcal{F}\big(u({\bs y})\big)\rho({\bs y})\d{\bs y}
= \sum_{\ell=0}^\infty\Delta Q_\ell\mathcal{F}\big(u({\bs y})\big)
= \sum_{\ell=0}^\infty\Delta Q_\ell\sum_{\ell'=0}^\infty\Delta\mathcal{F}_{\ell'}\big(u({\bs y})\big).
\]
Since \(\Delta Q_\ell\) is linear and continuous, we end up with
\[
\int_{{\Gamma}} \mathcal{F}\big(u({\bs y})\big)\rho({\bs y})\d{\bs y}
=\sum_{\ell,\ell'=0}^\infty\Delta Q_\ell\Delta\mathcal{F}_{\ell'}\big(u({\bs y})\big).
\]
Truncating this sum in accordance with \(\ell+\ell'\leq j\) then 
yields the multilevel quadrature representation \eqref{eq:nestedV} 
if we recombine the operators \(\Delta Q_\ell\). Analogously, we
obtain the representation \eqref{eq:nestedQ} if we recombine the 
operators \(\Delta\mathcal{F}_\ell\). Note that the sequence of 
the application of the operators \(\Delta Q_\ell\) and \(\Delta
\mathcal{F}_{\ell'}\) is crucial here. Moreover, we have repeatedly 
exploited the linearity of \(\Delta Q_\ell\).

Of course, the representations \eqref{eq:nestedV} and 
\eqref{eq:nestedQ} are mathematically equivalent. More precisely, 
if we set \(\mathcal{F}\big(\mathfrak{u}_{-1}({\bs y})\big)\isdef 0\),
there holds
\[
\sum_{\ell=0}^j Q_{j-\ell}\Delta \mathcal{F}_\ell\big(u({\bs y})\big)
=\sum_{\ell=0}^j \Delta Q_{\ell}\mathcal{F}\big(\mathfrak{u}_{j-\ell}({\bs y})\big).
\]
Thus, all available results for the representation \eqref{eq:nestedV} 
of the multilevel quadrature, see e.g.\ \cite{HPS1,HPS2} and the 
references therein, carry over to the representation \eqref{eq:nestedQ}.


Nonetheless, the multilevel quadrature based on 
representation \eqref{eq:nestedQ} has substantial advantages. 
On the one hand, it allows for an easy use of non-nested 
finite element meshes and even for adaptively refined finite 
element meshes. A further property of \eqref{eq:nestedQ} 
is an obvious reduction of the cost if nested quadrature 
formulae are employed.

\section{Error analysis}\label{sec:error}
In the sequel, we restrict ourselves for reasons of simplicity to the situations 
\(\mathcal{F}(u)=u\) and  \(\mathcal{F}(u)=u^2\)
which yield the expectation and the second moment of the solution to \eqref{eq:parprob}. This means that we consider
\begin{equation}\label{eq:sg-exp}
\Int u^p\approx
\sum_{\ell=0}^j \Delta Q_{\ell}\mathfrak{u}_{j-\ell}^p
=\sum_{\ell=0}^j Q_{j-\ell}\big(\mathfrak{u}_\ell^p-\mathfrak{u}_{\ell-1}^p\big)\quad\text{for }p=1,2.
\end{equation}
We derive a general approximation result for the multilevel quadrature based
on the generic estimate 
\begin{equation}\label{eq:generrest}
\big\|(\Int-{Q}_{\ell})(u^p-\mathfrak{u}^p_{\ell'})\big\|_{{\mathcal{X}}}
\lesssim 2^{-(\ell+\ell')}\|f\|_{L^2(D)}^p\quad\text{for }p=1,2
\end{equation}
with \(f\) being the right hand side of \eqref{eq:parprob} and \(h_{\ell'}\sim 2^{-\ell'}\).
In particular, any quadrature rule that satisfies this estimate gives rise to a
multilevel quadrature method.
In the sequel, we provide this estimate for the MLQMC as well as for the
multilevel Clenshaw-Curtis quadrature (MLCC). 

We remark that the derivation of the generic estimate \eqref{eq:generrest} 
for the Monte Carlo quadrature is straightforward under the condition that 
the integrand is square 
integrable with respect to the parameter \({\bs y}\), cf.~\cite{BSZ,HPS1}. 
In this case, the generic estimate can be derived similarly to Strang's lemma,
see \cite{TSU13}.
Nevertheless, since the Monte Carlo quadrature does not provide deterministic 
error estimates, we have to replace the norm in \(\mathcal{X}\) by the 
\(L^2_\rho({{\Gamma}};\mathcal{X})\)-norm. 


The situation becomes much more involved if parametric regularity has
to be taken into account. In the latter case, also bounds on the derivatives
of the increments have to be provided.
The next lemma is a generalization of similar results from \cite{HPS2,KSS3}, which provide
the smoothness of the Galerkin error with respect to the parameter \({\bs y}\in\Gamma\) 
for the non-conforming case \(D\neq D_\ell\).

\begin{lemma}\label{lem:decay}
For the error $\delta_\ell({\bs y})\isdef
(u-\mathfrak{u}_\ell)({\bs y})$ of the Galerkin projection, 
there holds the estimate
\begin{equation}
\big\|\partial^{\balpha}_{\bs y}\delta_\ell({\bs y})\big\|_{H^1(D)}\leq
Ch_\ell|\balpha|!
c^{|\balpha|}{\bgamma}^{\balpha}\|f\|_{L^2(D)}
\quad\text{for all }{\bs\alpha}\in\mathbb{N}^m,
\end{equation}
where \(\bgamma\isdef\{\gamma_k\}_{k=1}^m\), cf.~\eqref{eq:paramalpha}. 
The constants \(C,c>0\) are dependent on \(a_{\min}\) and \(a_{\max}\),
but independent of the parameter dimension \(m\).
\end{lemma}
\begin{proof}
By definition, there holds, cf.\ \eqref{eq:blf},
\[
\mathcal{B}_{\bs y}(\mathfrak{u}_\ell,v_\ell)
	= \int_Da({\bs y})\nabla\mathfrak{u}_\ell({\bs y})\nabla v_\ell\d{\bs x}
	= \int_D fv_\ell\d{\bs x}\quad\text{for all }v_\ell\in\mathcal{S}_\ell(D).
\]
On the other hand, integration by parts yields
\[
\mathcal{B}_{\bs y}(u,v_\ell)
	= \int_Da({\bs y})\nabla{u}({\bs y})\nabla v_\ell\d{\bs x}
	= \int fv_\ell\d{\bs x} + \int_{\partial D}a({\bs y})
		\frac{\partial u}{\partial {\bs n}}({\bs y})v_\ell({\bs x})\d\sigma_{\bs x}
\]
for all \(v_\ell\in S_\ell(D)\).
Thus, we obtain the perturbed Galerkin orthogonality
\begin{equation}\label{eq:pertGalerkin}
\mathcal{B}_{\bs y}\big(u-\mathfrak{u}_\ell,v_\ell) 
= \int_{\partial D}a({\bs y})\frac{\partial u}{\partial {\bs n}}({\bs y})
v_\ell({\bs x})\d\sigma_{\bs x}\quad\text{for all }v_\ell\in S_\ell(D).
\end{equation}
Due to the uniform ellipticity of the bilinear form, we can also define the 
Galerkin projection
\(\mathcal{P}_\ell({\bs y})\colon H^1_0(D)\to S_\ell(D)\) via
\[
\mathcal{B}_{\bs y}(u-\mathcal{P}_\ell u,v_\ell)=0\quad\text{for all }
v_\ell\in\mathcal{S}_\ell(D).
\]
It holds
\begin{equation}\label{eq:GalerkinSplit}
\begin{aligned}
\|\partial_{\bs y}^{\bs\alpha}(u-\mathfrak{u}_\ell)\|_{H^1(D)}
&\leq\|\mathcal{P}_\ell\partial_{\bs y}^{\bs\alpha}(u-\mathfrak{u}_\ell)\|_{H^1(D)}
+\|(I-\mathcal{P}_\ell)\partial_{\bs y}^{\bs\alpha}(u-\mathfrak{u}_\ell)\|_{H^1(D)}\\
&\leq\|\mathcal{P}_\ell\partial_{\bs y}^{\bs\alpha}(u-\mathfrak{u}_\ell)\|_{H^1(D)}
+\|(I-\mathcal{P}_\ell)\partial_{\bs y}^{\bs\alpha}u\|_{H^1(D)},
\end{aligned}
\end{equation}
since \(\partial_{\bs y}^{\bs\alpha}\mathfrak{u}_\ell\in S_\ell(D)\) and hence
\(\mathcal{P}_\ell\partial_{\bs y}^{\bs\alpha}\mathfrak{u}_\ell=\partial_{\bs y}^{\bs\alpha}\mathfrak{u}_\ell\).

In order to estimate the first term, we employ the perturbed Galerkin 
ortho\-go\-nality \eqref{eq:pertGalerkin} and obtain
\begin{equation}\label{eq:CDS}
\begin{aligned}
&\mathcal{B}_{\bs y}\big(\partial_{\bs y}^{\bs\alpha}(u-\mathfrak{u}_\ell),v_\ell\big)
-\partial^\balpha_{\bs y}\int_{\partial D}a({\bs y})
\frac{\partial u}{\partial {\bs n}}({\bs y})v_\ell\d\sigma_{\bs x}\\
&\qquad=-\sum_{\{k:\alpha_k\neq 0\}}\alpha_k\sqrt{\lambda_k}
\int_D\varphi_k\nabla\partial_{\bs y}^{\balpha-{\bs e}_k}(u-\mathfrak{u}_\ell)({\bs y})
\nabla v_\ell\d{\bs x},
\end{aligned}
\end{equation}
see e.g.\ \cite{CDS}. The derivatives of the boundary term satisfy
\begin{align*}
&\partial^\balpha_{\bs y}\int_{\partial D}a({\bs y})\frac{\partial u}{\partial {\bs n}}({\bs y})v_\ell\d\sigma_{\bs x}
=\sum_{\balpha'\leq\balpha}{\balpha\choose\balpha'}
\int_{\partial D}\big[\partial^{\balpha'}_{\bs y}a({\bs y})\big]\bigg[\partial^{\balpha-\balpha'}_{\bs y}\frac{\partial u}{\partial {\bs n}}({\bs y})\bigg]v_\ell\d\sigma_{\bs x}\\
&\quad=\int_{\partial D}a({\bs y})\frac{\partial(\partial^{\balpha}_{\bs y}u)}
	{\partial {\bs n}}({\bs y})v_\ell\d\sigma_{\bs x}
+\sum_{\{k:\alpha_k\neq 0\}}\alpha_k\sqrt{\lambda_k}
	\int_{\partial D}\varphi_k\frac{\partial(\partial^{\balpha-{\bs e}_k}_{\bs y}u)}
		{\partial {\bs n}}({\bs y})v_\ell\d\sigma_{\bs x}.
\end{align*}
Inserting this identity into \eqref{eq:CDS} yields
\begin{equation}\label{eq:CDSext}
\begin{aligned}
&\mathcal{B}_{\bs y}\big(\partial_{\bs y}^{\bs\alpha}(u-\mathfrak{u}_\ell),v_\ell\big)-
\int_{\partial D}a({\bs y})
\frac{\partial(\partial^\balpha_{\bs y}u)}{\partial {\bs n}}({\bs y})v_\ell\d\sigma_{\bs x}\\
&\qquad=-\sum_{\{k:\alpha_k\neq 0\}}\alpha_k\sqrt{\lambda_k}
\bigg[\int_D\varphi_k\nabla\partial_{\bs y}^{\balpha-{\bs e}_k}(u-\mathfrak{u}_\ell)({\bs y})
\nabla v_\ell\d{\bs x}\\
&\hspace*{5cm}-\int_{\partial D}\varphi_k\frac{\partial(\partial^{\balpha-{\bs e}_k}_{\bs y}u)}
	{\partial {\bs n}}({\bs y})v_\ell\d\sigma_{\bs x}\bigg].
\end{aligned}
\end{equation}

In order to bound the boundary integrals, we employ 
the following estimate, which holds true for any \(v,w\in H^1(D)\). It holds
\begin{align*}
\bigg|\int_{\partial D}a({\bs y})\frac{\partial v}{\partial {\bs n}}w\d\sigma_{\bs x}\bigg|
&\leq a_{\max} \bigg\|\frac{\partial v}{\partial {\bs n}}({\bs y})\bigg\|_{H^{-1/2}(\partial D)}
\|w\|_{H^{1/2}(\partial D)}\\
&\leq C_{\text{inv}}a_{\max}\|v\|_{H^1(D)}\|w\|_{H^{1/2}(\partial D)},
\end{align*}
where \(C_{\operatorname{inv}}\) is the norm of the inverse Neumann 
trace operator. Next, we employ a discrete version of the trace theorem 
provided by \cite[Lemma I\!I\!I.1.6]{B}, which reads
\begin{equation}\label{eq:discTrace}
\|v_\ell\|_{H^{1/2}(\partial D)}\leq c h_\ell\|v_\ell\|_{H^1(D)}\quad\text{for all }
v_\ell\in\mathcal{S}_\ell(D)
\end{equation}
with some constant \(c>0\). From this, we infer 
\[
\bigg|\int_{\partial D}a({\bs y})\frac{\partial(\partial^{\bs\alpha}_{\bs y}u)}{\partial {\bs n}}({\bs y})v_\ell\d\sigma_{\bs x}\bigg|
\leq Ch_\ell\|\partial^{\bs\alpha}_{\bs y}u({\bs y})\|_{H^1(D)}\|v_\ell\|_{H^{1}(D)}
\]
for all \(v_\ell\in\mathcal{S}_\ell(D)\) and some constant \(C>0\).

Inserting the latter estimate into \eqref{eq:CDSext} and choosing 
\(v_\ell=\mathcal{P}_\ell\partial_{\bs y}^{\bs\alpha}(u-\mathfrak{u}_\ell)\)
as test function, we arrive at
\begin{align*}
&a_{\min}\|\mathcal{P}_\ell\partial_{\bs y}^{\bs\alpha}(u-\mathfrak{u}_\ell)({\bs y})\|_{H^1(D)}^2
\leq Ch_\ell\|\mathcal{P}_\ell\partial_{\bs y}^{\bs\alpha}u({\bs y})\|_{H^1(D)}
\|\mathcal{P}_\ell\partial_{\bs y}^{\bs\alpha}(u-\mathfrak{u}_\ell)({\bs y})\|_{H^{1}(D)}\\
&\qquad\phantom{\leq}
+\sum_{\{k:\alpha_k\neq 0\}}\alpha_k\gamma_k\Big[
\|\partial_{\bs y}^{\balpha-{\bs e}_k}(u-\mathfrak{u}_\ell)({\bs y})\|_{H^1(D)}
\|\mathcal{P}_\ell\partial_{\bs y}^{\bs\alpha}(u-\mathfrak{u}_\ell)({\bs y})\|_{H^{1}(D)}\\
&\hspace*{4cm}+Ch_\ell\|\partial_{\bs y}^{\balpha-{\bs e}_k}u(\bs y)\|_{H^1(D)}
\|\mathcal{P}_\ell\partial_{\bs y}^{\bs\alpha}(u-\mathfrak{u}_\ell)({\bs y})\|_{H^{1}(D)}\Big].
\end{align*}
Simplifying this expression yields
\begin{align*}
&\|\mathcal{P}_\ell\partial_{\bs y}^{\bs\alpha}(u-\mathfrak{u}_\ell)({\bs y})\|_{H^1(D)}
\leq Ch_\ell\|\partial_{\bs y}^{\bs\alpha}u({\bs y})\|_{H^1(D)}\\
&\qquad+C\!\!\!\!\sum_{\{k:\alpha_k\neq 0\}}\alpha_k\gamma_k\Big[
\|\partial_{\bs y}^{\balpha-{\bs e}_k}(u-\mathfrak{u}_\ell)({\bs y})\|_{H^1(D)}
+ h_\ell\|\partial_{\bs y}^{\balpha-{\bs e}_k}u(\bs y)\|_{H^1(D)}\Big]
\end{align*}
for some other constant \(C>0\),
where we employed the stability of the Galerkin projection in the first term.
Next, in view of the estimate
\[
\|\partial_{\bs y}^{\bs\alpha}u({\bs y})\|_{H^1(D)}\le C
	|\bs\alpha|!c^{|\bs\alpha|}{\bs\gamma}^{\bs\alpha} \|f\|_{L^2(D)}
\]
for some constants \(C,c>0\),
see \cite{CDS}, we end up with
\begin{align*}
&\|\mathcal{P}_\ell\partial_{\bs y}^{\bs\alpha}(u-\mathfrak{u}_\ell)({\bs y})\|_{H^1(D)}\\
&\qquad\leq Ch_\ell c^{|\bs\alpha|}|\bs\alpha|!{\bs\gamma}^{\bs\alpha} \|f\|_{L^2(D)}
+C\!\!\!\!\sum_{\{k:\alpha_k\neq 0\}}\alpha_k\gamma_k
\|\partial_{\bs y}^{\balpha-{\bs e}_k}(u-\mathfrak{u}_\ell)({\bs y})\|_{H^1(D)}.
\end{align*}
for some constants \(C,c>0\).
Combining this with the initial estimate \eqref{eq:GalerkinSplit}
gives then
\begin{align*}
&\|\partial_{\bs y}^{\bs\alpha}(u-\mathfrak{u}_\ell)({\bs y})\|_{H^1(D)}
\leq C\!\!\!\!\sum_{\{k:\alpha_k\neq 0\}}\alpha_k\gamma_k
\|\partial_{\bs y}^{\balpha-{\bs e}_k}(u-\mathfrak{u}_\ell)({\bs y})\|_{H^1(D)}\\
&\hspace*{4.5cm}+Ch_\ell c^{|\bs\alpha|}|\bs\alpha|!{\bs\gamma}^{\bs\alpha} \|f\|_{L^2(D)}+
\|(I-\mathcal{P}_\ell)\partial_{\bs y}^{\bs\alpha}u\|_{H^1(D)}\\
&\hspace*{2cm}\leq C\!\!\!\!\sum_{\{k:\alpha_k\neq 0\}}\alpha_k\gamma_k
\|\partial_{\bs y}^{\balpha-{\bs e}_k}(u-\mathfrak{u}_\ell)({\bs y})\|_{H^1(D)}
	+Ch_\ell c^{|\bs\alpha|}|\bs\alpha|!{\bs\gamma}^{\bs\alpha} \|f\|_{L^2(D)},
\end{align*}
where we used 
\(\|(I-\mathcal{P}_\ell)\partial_{\bs y}^{\bs\alpha}u\|_{H^1(D)}
\leq C h_\ell c^{|\bs\alpha|}|\bs\alpha|!{\bs\gamma}^{\bs\alpha}\|f\|_{L^2(D)}\)
for some constants \(C,c>0\),
which follows from the approximation property of the finite element 
space \(S_\ell(D)\) and \cite[Theorem 6]{KSS3}. The proof is now 
concluded similarly to the proof of \cite[Theorem 7]{KSS3}.
\end{proof}

With this lemma, it is easy to show the following
result related to the second moment, cf.~\cite{HPS2}.
\begin{lemma}\label{lem:decderp}
The derivatives of the difference \(u^2-\mathfrak{u}^2_{\ell}\) 
satisfy the estimate 
\begin{equation}\label{eq:reggalhigh}
\big\|\partial^{\balpha}_{\bs y}\big({u}^{2}-\mathfrak{u}^2_{
\ell}\big)({\bs y})\big\|_{W^{1,1}(D)}\leq C
h_\ell |\balpha|! c^{|\balpha|}\bgamma^{\balpha} \|f\|_{L^2(D)}^2
\quad\text{for all }{\bs\alpha}\in\mathbb{N}^m
\end{equation}
with constants \(C,c>0\) dependent on \(a_{\min}\) and \(a_{\max}\).
\end{lemma}

With the aid of Lemmata~\ref{lem:decay} and \ref{lem:decderp} 
together with the results from \cite{Wan02}, the generic
error estimate for the MLQMC with Halton points can be derived. The next lemma is 
for example shown in \cite{HPS16,Sieb15}.

\begin{lemma}\label{lem:QMCerr} Let \(u\in L^2_\rho\big({{\Gamma}};H^1_0(D)\big)\) 
be the solution to \eqref{eq:parprob} and 
\(\mathfrak{u}_\ell\) the associated Galerkin projection on 
level \(\ell\). Moreover, let
\(\rho_k\in W^{1,\infty}(-1,1)\) for \(k=1,\ldots,m\).
Then, for the quasi-Monte Carlo quadrature based on Halton points, there holds
\begin{equation}\label{eq:ErrQMC}
\big\|(\Int-{Q}_{\ell})(u^p-\mathfrak{u}^p_{\ell'})\big\|_{\mathcal{X}}
\lesssim 2^{-(\ell+\ell')}\|f\|_{L^2(D)}^p\quad\text{for }p=1,2
\end{equation}
with \(N_\ell\sim 2^{\ell/(1-\delta)}\) for arbitrary \(\delta>0\).
\end{lemma}

The next lemma establishes the generic estimate 
for the sparse grid quadrature based on the nested 
\emph{Clenshaw-Curtis abscissae}, cf.~\cite{GG,NR96}. These 
are given by the extrema of the Chebyshev polynomials
\[
\xi_k=\cos\bigg(\frac{(k-1)\pi}{n-1}\bigg)\quad\text{ for }k=1,\ldots,n, 
\]
where \(n=2^{j-1}+1\) if \(j>1\) and \(n=1\) with \(\xi_1=0\) if \(j=1\).
\begin{lemma}\label{lem:SGQerr}  
Let \(u\in L^2_\rho\big({{\Gamma}};H^1_0(D)\big)\) be the solution to \eqref{eq:parprob}
and let \(\mathfrak{u}_\ell\) be the associated Galerkin projection on level \(\ell\). Moreover, 
let \(\rho_k(y_k)\in W^{r,\infty}(-1,1)\) for \(k=1,\ldots,m\). Then, for the sparse grid 
quadrature based on Clenshaw-Curtis abscissae, there holds
\begin{equation}\label{eq:ErrSGQ}
\big\|(\Int-{Q}_{\ell})(u^p-\mathfrak{u}^p_{\ell'})\big\|_{\mathcal{X}}
\lesssim 2^{-(\ell r+\ell')}\ell^{m-1}\|f\|_{L^2(D)}^p\quad\text{for }p=1,2
\end{equation}
provided that \(N_\ell\sim 2^\ell\ell^{d-1}\).
\end{lemma}
 
\begin{proof}
It is shown in \cite{NR96} that the number \(N_\ell\) of quadrature 
points of the sparse tensor product quadrature with 
Clenshaw-Curtis abscissae is of the order \(\mathcal{O}(2^\ell\ell^{d-1})\). In addition, 
we have for functions \(v\colon{{\Gamma}}\to\mathbb{R}\) with mixed 
regularity the following error bound:
\[
{\bigg|\int_{\Gamma}v({\bs y})\d{\bs y}-\sum_{i=1}^{N_\ell}w_iv(\boldsymbol{\xi}_i)\bigg|}\lesssim 2^{-\ell r}\ell^{(m-1)}\max_{\|\balpha\|_\infty\leq r}
	\big\|\partial^\balpha_{\bs y}v\big\|_{L^\infty({{\Gamma}})}.
\]
Hence, to prove the desired assertion, we have to provide estimates on the 
derivatives \(\partial^\balpha_{\bs y}{\big[}\big(u^p({\bs y})-\mathfrak{u}^p_{{\ell'}}({\bs y})\big)
\rho({\bs y}){\big]}\). This can be accomplished by the Leibniz formula as 
in the proof of the previous lemma:
\begin{align*}
&\big\|\partial^\balpha_{\bs y}\big[({u}^p-{\mathfrak{u}}^p_{\ell'})({\bs y})\rho({\bs y})\big]\big\|_{\mathcal{X}}\\
&\qquad\leq\sum_{\balpha'\leq\balpha}{\balpha\choose\balpha'}\big\|\partial^{\balpha-\balpha'}_{\bs y}({u}^p-{\mathfrak{u}}^p_{\ell'})({\bs y})\big\|_{\mathcal{X}}
\big\|\partial^{\balpha'}_{\bs y}\rho({\bs y})\big\|_{L^\infty({{\Gamma}})}\\
&\qquad\lesssim
2^{-\ell'}\sum_{\balpha'\leq\balpha}{\balpha\choose\balpha'}
 |\balpha-\balpha'|! c^{|\balpha-\balpha'|}\bgamma^{\balpha-\balpha'} \|f\|_{L^2(D)}^p
\brho^{\balpha'}\\
&\qquad \lesssim 2^{-\ell'}(|\balpha|+1)!\|f\|_{L^2(D)}^p\tilde{c}^{|\balpha|}.
\end{align*}
Herein, we introduced again the quantity \(\brho\isdef
\big[\|\rho_1\|_{W^{r,\infty}(-1,1)},\ldots,\|\rho_m\|_{W^{r,\infty}(-1,1)}\big]\)
and \(\tilde{c}=\max_{k=1,\ldots,m}\max\{c\gamma_k,\rho_k\}\).
We set \(C(r)\isdef\max_{\|\balpha\|_\infty\leq r}(|\balpha|+1)!\tilde{c}^{|\balpha|}\)
and obtain
\begin{align*}
\big\|(\Int-{Q}_{\ell})(u^p-\mathfrak{u}^p_{\ell'})\big\|_{\mathcal{X}}^2
&\lesssim\big(2^{-\ell r}\ell^{(m-1)}2^{-\ell'}C(r)\|f\|_{L^2(D)}^p\big)^2.
\end{align*}
Then, exploiting that the bound on the derivatives of the integrand is independent of the 
parameter and taking square roots on both sides
completes the proof.
\end{proof}

\begin{remark}
As for the quasi-Monte Carlo quadrature,
by slightly decreasing \(r\) in the convergence result for the 
sparse tensor product quadrature, we may remove the factor 
\(\ell^{m-1}\) since \(\ell^{m-1}\lesssim 2^{\ell\delta}\) for arbitrary \(\delta>0\).
\end{remark}

Estimates of the type \eqref{eq:generrest} are crucial to show the
following approximation result for the multilevel quadrature.
More general, every quadrature that satisfies an estimate of
type \eqref{eq:generrest} is feasible for a related multilevel quadrature
method.

\begin{theorem}\label{theo:MLest}
Let $\{{Q}_\ell\}$ be a sequence of quadrature rules 
that satisfy an estimate of type \eqref{eq:generrest}, where
\(u\in L^2_\rho\big({{\Gamma}},H^1_0(D)\big)\) is the solution to 
\eqref{eq:parprob} that satisfies \eqref{eq:err-est-1}. 
Then, the error of the multilevel estimator for the mean
and the second moment defined 
in \eqref{eq:sg-exp}
is bounded by
\begin{equation}\label{eq:errormlest}
\bigg\|\Int{u}^p-\sum_{\ell=0}^j\Delta{Q}_{\ell}
\mathfrak{u}^p_{j-\ell}\bigg\|_{\mathcal{X}}
   		\lesssim 2^{-j}j\|f\|^p_{L^{2}(D)},
\end{equation}
where \(\mathcal{X}=H^1(D)\) if \(p=1\) and \(\mathcal{X}=W^{1,1}(D)\) if \(p=2\). 
\end{theorem}

\begin{proof}
We shall apply the following multilevel splitting of the error 
\begin{equation}\label{eq:multsplitting}
\begin{aligned}
\bigg\|\Int{u}^p-\sum_{\ell=0}^j\Delta{Q}_{\ell}
\mathfrak{u}^p_{j-\ell}\bigg\|_{{\mathcal{X}}}
&=\bigg\|\Int{u}^p-Q_ju^p+\sum_{\ell=0}^j\Delta Q_\ell u^p-\sum_{\ell=0}^j\Delta{Q}_{\ell}
\mathfrak{u}^p_{j-\ell}\bigg\|_{\mathcal{X}}\\
&\le\big\|\Int{u}^p-Q_ju^p\big\|_{\mathcal{X}}
	+ \sum_{\ell=0}^j\big\|\Delta Q_{\ell}\big({u}^p-\mathfrak{u}^p_{j-\ell}\big)\big\|_{\mathcal{X}}. 
\end{aligned}
\end{equation}
The first term just reflects the quadrature error and can be bounded
with similar arguments as in Lemmata~\ref{lem:QMCerr} and \ref{lem:SGQerr} 
according to
\[
\big\|\Int{u}^p-Q_ju^p\big\|_{\mathcal{X}}\lesssim 2^{-j}\|f\|_{L^2(D)}^p
\]
with a constant that depends on \(m\).
The term inside the sum satisfies with \eqref{eq:generrest} that
\begin{equation*}
\begin{aligned}
\big\|\Delta Q_{\ell}\big({u}^p-\mathfrak{u}^p_{j-\ell}\big)\big\|_{\mathcal{X}}
&\leq
\big\|(\Int-Q_{\ell})\big({u}^p-\mathfrak{u}^p_{j-\ell}\big)\big\|_{\mathcal{X}}
+\big\|(\Int-Q_{\ell-1})\big({u}^p-\mathfrak{u}^p_{j-\ell}\big)\big\|_{\mathcal{X}}\\
&\lesssim
2^{-(\ell+j-\ell)}\|f\|_{L^2(D)}^p+2^{-(\ell-1+j-\ell)}\|f\|_{L^2(D)}^p\lesssim 2^{-j}\|f\|_{L^2(D)}^p.
\end{aligned}
\end{equation*}
Thus, we can estimate \eqref{eq:multsplitting} as
\[
\bigg\|\Int{u}^p-\sum_{\ell=0}^j\Delta{Q}_{\ell}
\mathfrak{u}^p_{j-\ell}\bigg\|_{\mathcal{X}}\lesssim 2^{-j}\|f\|_{L^2(D)}^p
+\sum_{\ell=0}^j2^{-j}\|f\|_{L^2(D)}^p
\leq 2^{-j}(j+2)\|f\|_{L^2(D)}^p.
\]
This completes the proof.
\end{proof}
\begin{remark}
Note that we can achieve in our framework also
nestedness for the samples in the Monte Carlo method.
This is due to the fact that independent samples have
to be used only for the estimators \(Q_\ell\) for \(\ell=0,\ldots,j\).
But from the proof of the previous theorem, we see that \(Q_\ell\)
has not to be sampled independently from \(Q_{\ell'}\) for \(\ell\neq\ell'\).
Thus, we may employ the same underlying set of sample points on each
level.
\end{remark}

\section{Numerical approximation}
\label{sec:consistency}
The previous results guarantee that the consistency error due to 
the non-confor\-mity of the finite element space is of the correct order. 
In the actual implementation, instead of considering the bilinear form 
introduced in \eqref{eq:blf}, we shall consider on level \(\ell\geq 0\) 
the variational formulation
\[
\int_{D_\ell}\tilde{a}_\ell({\bs x},{\bs y})\nabla{\tilde{\mathfrak{u}}}_\ell
	\nabla v_\ell\d{\bs x}=\int_{D_\ell}fv_\ell\d{\bs x}
		\quad\text{for all }v_\ell\in\mathcal{S}^1_\ell(D),
\]
where \(\tilde{a}_\ell({\bs x},{\bs y})\) is a suitable piecewise 
constant approximation of \(a({\bs x},{\bs y})\) with respect to the 
triangulation $\mathcal{T}_\ell$ on \(D_\ell\). In this section,
we will provide a result that also takes into account the consistency 
error due to numerical quadrature in the bilinear form.
In particular, we account for the quadrature error that is 
introduced by integration with respect to \(D_\ell\) instead of integration 
with respect to \(D\). 

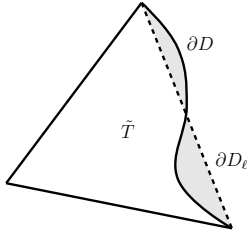
\begin{figure}[htb]
\begin{center}
\scalebox{0.6}{
\begin{tikzpicture}
\path[fill=gray] (3,4)--(5,-1) 
plot [smooth, tension=0.7] coordinates {(5,-1) (3.8,0.1) (4,1.7) (3.8,3) (3,4)};
\draw[line width = 1.5pt] plot [smooth, tension=0.7] coordinates {(5,-1) (3.8,0.1) (4,1.7) (3.8,3) (3,4)};
\draw[line width = 1.5pt] (0,0)-- (5,-1);
\draw[line width = 1.5pt] (0,0)-- (3,4);
\draw[line width = 1.5pt, dashed] (3,4)--(5,-1);
\draw (4.3,3.1) node {\Large\(\partial D\)};
\draw (5.0,0.5) node {\Large\(\partial D_\ell\)};
\draw (2.7,1.2) node {\Large\(\tilde{T}\)};
\end{tikzpicture}}
\end{center}
\caption{\label{fig:triangle}Triangle \(\tilde{T}\) located at the boundary of the domain. The solid
line indicates the boundary of \(D\), while the dashed line indicates the boundary of \(D_\ell\).}
\end{figure}

The situation is sketched in Fig.~\ref{fig:triangle} 
for the two dimensional case: For the given triangle \(T\) at the domain's 
boundary, the areas of the true domain \(D\) and its polygonal approximation
\(D_\ell\) differ by the grey shaded area. According to \cite{B}, this area is 
small relative to the size of the element. There holds
\begin{equation}\label{eq:triangest}
|\tilde{T}\cap(D\triangle D_\ell)|\leq ch_\ell|\tilde{T}|\quad\text{for some constant }c>0,
\end{equation}
where \(D\triangle D_\ell\isdef (D\setminus D_\ell)\cup(D\setminus D_\ell)\)
is the symmetric difference of sets. 
Moreover, since we consider piecewise linear finite elements which are set to
zero outside of \(D_\ell\), we have
\[
\int_{\tilde{T}}a({\bs y})\nabla\mathfrak{u}_\ell({\bs y})\nabla v_\ell\d{\bs x}=
\nabla\mathfrak{u}_\ell({\bs y})|_T\nabla v_\ell|_T\int_{\tilde{T} \cap{T}} a({\bs y})\d{\bs x},
\]
where \({T}\in\mathcal{T}_\ell\) is the polygonal approximation to \(\tilde{T}\).
Hence, setting
\[
a_\ell({\bs y})|_{T\cup\tilde{T}}\isdef\frac{1}{|T|}\int_{\tilde{T}\cap{T}}a({\bs y})\d{\bs x}
\]
yields
\[
\int_{\tilde{T}}a({\bs y})\nabla\mathfrak{u}_\ell({\bs y})\nabla v_\ell\d{\bs x}
=\int_{{T}}a_\ell({\bs y})\nabla\mathfrak{u}_\ell({\bs y})\nabla v_\ell\d{\bs x}
\quad\text{for all }T\in\mathcal{T}_\ell, v_\ell\in\mathcal{S}^1_\ell(D)
\]
and, therefore,
\[
\int_D a({\bs y})\nabla\mathfrak{u}_\ell({\bs y})\nabla v_\ell\d{\bs x}
=\int_{D_\ell}a_\ell({\bs y})\nabla\mathfrak{u}_\ell({\bs y})\nabla v_\ell\d{\bs x}
\quad\text{for all }v_\ell\in\mathcal{S}^1_\ell(D).
\]
Nevertheless, for numerical computations, it is 
more convenient to 
assume that \(a({\bs y})\in C^{0,1}(D\cup D_\ell)\) for all \(\ell\geq 0\) and
the barycenter \({\bs x}_c\in {T}\) is also contained in \(\tilde{T}\). Then, 
to avoid integration with respect to the curved element \(\tilde{T}\),
we employ a midpoint rule and consider $\tilde{a}_\ell({\bs y})|_{T\cup\tilde{T}}
\isdef a({\bs x}_c,{\bs y})$ instead. We have the following

\begin{lemma} 
There holds
\[
\big\|\partial_{\bs y}^\balpha(a_\ell-\tilde{a}_\ell)({\bs y})\big\|_{L^\infty(D)}\leq ch_\ell\bgamma^\balpha\|a({\bs y})\|_{W^{1,\infty}(D)}
\]
for some constant \(c>0\) which depends on \eqref{eq:triangest}.
\end{lemma}
\begin{proof}
By Taylor's theorem, there holds
\begin{equation}\label{eq:TaylorCoeff}
\|a({\bs y}) - a({\bs x}_c,{\bs y})\|_{L^\infty(D)}\leq ch_\ell\|a({\bs y})\|_{W^{1,\infty}(D)}.
\end{equation}
Moreover, we note that \(a_\ell\) as well as \(\tilde{a}_\ell\)  
differ on at most on \(|\mathcal{T}_\ell|\) triangles, where the 
difference is constant for each \(T\in\mathcal{T}_\ell\). Hence, we obtain
\begin{align*}
&\big\|\partial_{\bs y}^\balpha\big(a_\ell-\tilde{a}_\ell\big)({\bs y})\big\|_{L^\infty(D)}
=\max_{T\in\mathcal{T}_\ell}\frac{1}{|T|}\bigg|\int_{\tilde{T}\cap T}\partial_{\bs y}^\balpha 
a({\bs y})\d{\bs x}-\int_{T}\partial_{\bs y}^\balpha a({\bs x}_c,{\bs y})\d{\bs x}\bigg|\\
&\qquad=\max_{T\in\mathcal{T}_\ell}\frac{1}{|T|}\bigg|\int_{\tilde{T}\cap T}\partial_{\bs y}^\balpha\big(a-a({\bs x}_c)\big)({\bs y})\d{\bs x}-\int_{T\setminus\tilde{T}}\partial_{\bs y}^\balpha a({\bs x}_c,{\bs y})\d{\bs x}\bigg|\\
&\qquad\leq \max_{T\in\mathcal{T}_\ell}\frac{1}{|T|}\bigg(\bigg|\int_{\tilde{T}\cap T}\partial_{\bs y}^\balpha\big(a-a({\bs x}_c)\big)({\bs y})\d{\bs x}\bigg|+\bigg|\int_{T\setminus\tilde{T}}\partial_{\bs y}^\balpha a({\bs x}_c,{\bs y})\d{\bs x}\bigg|\bigg).
\end{align*}
Obviously, since \(a({\bs y})\) as well as \(a({\bs x}_c,{\bs y})\) are affine functions with respect to \({\bs y}\), all derivatives
for \(|\balpha|>1\) vanish. For \(|\balpha|\leq 1\), the first term is estimated by \eqref{eq:TaylorCoeff} together with the
fact that \(|T|=|T\cap\tilde{T}|\big(1+\mathcal{O}(h_\ell)\big)\), while the second term can be bounded by \(h_\ell\gamma_k\|a({\bs y})\|_{W^{1,\infty}(D)}\) if \(\alpha_k=1\), due to \eqref{eq:triangest}.
 Consequently, we obtain
\[
\big\|\partial_{\bs y}^\balpha\big(a_\ell-\tilde{a}_\ell\big)({\bs y})\big\|_{L^\infty(D)}\leq
\begin{cases}c h_\ell \|a({\bs y})\|_{W^{1,\infty}(D)}, &|\balpha|=0,\\
 c h_\ell\gamma_k\|a({\bs y})\|_{W^{1,\infty}(D)}, & \alpha_k = 1,\\
 0, & |\balpha|>1,
\end{cases}\]
for some constant \(c>0\). This completes the proof.
\end{proof}

Having this lemma at our disposal, we can prove the main result of this section.
\begin{theorem}\label{thm:fullydisc}
Let \(\mathfrak{u}_\ell\in\mathcal{S}_\ell(D)\) be the solution to 
\[
\int_{D_\ell}a_\ell({\bs y})\nabla\mathfrak{u}_\ell\nabla v_\ell\d{\bs x}=\int_{D_\ell}fv_\ell\d{\bs x}
\quad\text{for all }v_\ell\in\mathcal{S}_\ell(D),
\]
while \(\tilde{\mathfrak{u}}_\ell\in\mathcal{S}_\ell(D)\) solves 
\[
\int_{D_\ell}\tilde{a}_\ell({\bs y})\nabla\tilde{\mathfrak{u}}_\ell\nabla v_\ell\d{\bs x}=\int_{D_\ell}fv_\ell\d{\bs x}
\quad\text{for all }v_\ell\in\mathcal{S}_\ell(D).
\]
Then, there holds
\[
\|\partial^{\bs\alpha}_{\bs y}({\mathfrak{u}}_\ell-\tilde{\mathfrak{u}}_\ell)({\bs y})\|_{H^1(D)}
	\leq C h_\ell |c|^{|\bs\alpha|}|\bs\alpha|!{\bs\gamma}^{\bs\alpha}\|a({\bs y})\|_{W^{1,\infty}(D)}
		\|\tilde{\mathfrak{u}}({\bs y})\|_{H^1(D)}
\]
for some constants \(C,c>0\), which are independent of the parameter dimension \(m\).
\end{theorem}
\begin{proof}
There holds 
\[
\int_{D_\ell}{a}_\ell({\bs y})\nabla\big({\mathfrak{u}}_\ell-\tilde{\mathfrak{u}}_\ell\big)(\bs {y})\nabla v_\ell\d{\bs x}
=\int_{D_\ell}(\tilde{a}_\ell-{a}_\ell)({\bs y})\nabla\tilde{\mathfrak{u}}_\ell(\bs {y})\nabla v_\ell\d{\bs x}.
\]
Differentiating this equation yields via the Leibniz formula
\begin{align*}
&\int_{D_\ell}{a}_\ell({\bs y})\nabla\partial^{\bs\alpha}_{\bs y}\big({\mathfrak{u}}_\ell-\tilde{\mathfrak{u}}_\ell\big)(\bs {y})\nabla v_\ell\d{\bs x}\\
&\qquad=-\sum_{\{k:\alpha_k\neq 0\}}\alpha_k\int_{D_\ell}\partial^{{\bs e}_k}{a}_\ell({\bs y})
\nabla\partial^{{\bs\alpha}-{\bs e}_k}_{\bs y}\big({\mathfrak{u}}_\ell-\tilde{\mathfrak{u}}_\ell\big)(\bs {y})\nabla v_\ell\d{\bs x}\\
&\qquad\quad+\int_{D_\ell}(\tilde{a}_\ell-{a}_\ell)({\bs y})\nabla\partial^{\bs\alpha}_{\bs y}\tilde{\mathfrak{u}}_\ell(\bs {y})\nabla v_\ell\d{\bs x}\\
&\qquad\quad+\sum_{\{k:\alpha_k\neq 0\}}\alpha_k\int_{D_\ell}\partial^{{\bs e}_k}(\tilde{a}_\ell-a_\ell)({\bs y})
\nabla\partial^{{\bs\alpha}-{\bs e}_k}_{\bs y}\tilde{\mathfrak{u}}_\ell(\bs {y})\nabla v_\ell\d{\bs x}.
\end{align*}

Hence, choosing \(v_\ell=\partial^{\bs\alpha}_{\bs y}({\mathfrak{u}}_\ell-\tilde{\mathfrak{u}}_\ell\big)(\bs {y})\) results in
\begin{align*}
&a_{\ell,\min}\|\partial^{\bs\alpha}_{\bs y}({\mathfrak{u}}_\ell-\tilde{\mathfrak{u}}_\ell\big)({\bs y})\|_{H^1(D)}
\leq\sum_{\{k:\alpha_k\neq 0\}}\alpha_k\gamma_k\|\partial^{{\bs\alpha}-{\bs e}_k}_{\bs y}({\mathfrak{u}}_\ell-\tilde{\mathfrak{u}}_\ell)(\bs {y})\|_{H^1(D)}\\
&\hspace*{3cm}+ch_\ell\|a({\bs y})\|_{W^{1,\infty}(D)}
\|\partial^{\bs\alpha}_{\bs y}\tilde{\mathfrak{u}}_\ell({\bs y})\|_{H^1(D)}\\
&\hspace*{3cm}+\sum_{\{k:\alpha_k\neq 0\}}\alpha_k c\gamma_k h_\ell\|a({\bs y})\|_{W^{1,\infty}(D)}
\|\partial^{{\bs\alpha}-{\bs e}_k}_{\bs y}\tilde{\mathfrak{u}}_\ell({\bs y})\|_{H^1(D)},
\end{align*}
where \(a_{\ell,\min}>0\) is the constant of ellipticity associated to \(a_\ell\).

Next, we note that the standard bootstrapping argument can be employed to obtain the estimate
\[
\|\partial_{\bs y}^{\bs\alpha}\tilde{\mathfrak{u}}({\bs y})\|_{H^1(D)}\leq C
	|\bs\alpha|!c^{|\bs\alpha|}{\bs\gamma}^{\bs\alpha} \|\tilde{\mathfrak{u}}({\bs y})\|_{H^1(D)}
\]
for some constants \(C,c>0\), see e.g.\ \cite{CDS}. Therefore, we arrive at
\begin{align*}
&a_{\ell,\min}\|\partial^{\bs\alpha}_{\bs y}({\mathfrak{u}}_\ell-\tilde{\mathfrak{u}}_\ell\big)({\bs y})\|_{H^1(D)}
\leq\sum_{\{k:\alpha_k\neq 0\}}\alpha_k\gamma_k\|\partial^{{\bs\alpha}-{\bs e}_k}_{\bs y}
	({\mathfrak{u}}_\ell-\tilde{\mathfrak{u}}_\ell)(\bs {y})\|_{H^1(D)}\\
&\hspace*{5.5cm}+C|\bs\alpha|!h_\ell c^{|\bs\alpha|}{\bs\gamma}^{\bs\alpha}\|a({\bs y})\|_{W^{1,\infty}(D)}
	\|\tilde{\mathfrak{u}}({\bs y})\|_{H^1(D)}.
\end{align*}
From the previous estimate, the claim is again obtained as in the proof of \cite[Theorem 7]{KSS3}.
\end{proof}

The theorem directly yields to the fully discrete generic estimate
\[
\big\|(\Int-{Q}_{\ell})(u^p-\tilde{\mathfrak{u}}^p_{\ell'})\big\|_{{\mathcal{X}}}
\lesssim 2^{-(\ell+\ell')}\|f\|_{L^2(D)}^p\quad\text{for }p=1,2
\]
by using \eqref{eq:generrest} and the triangle inequality.
\section{Numerical results}\label{sec:results}
The numerical examples in this section are performed in three spatial 
dimensions. For the finite element discretization, we employ {\sc Matlab} 
and the Partial Differential Equation Toolbox\footnote{Release 2015a, The 
MathWorks, Inc., Natick, Massachusetts, United States.}. In both examples, 
the error is measured by interpolating the obtained solutions on a sufficiently
fine grid and comparing it there to a reference solution. We consider 
the MLMC, the MLQMC based on the Halton sequence, and the MLCC. 
Moreover, we set the density to \(\rho({\bs y})=(1/2)^m\) for our problems.

\subsection{An analytical example}
With our first example, we intend to validate the proposed method. To 
this end, we consider a simple quadrature problem on the unit ball
\(D=\{{\bs x}\in\mathbb{R}^3:\|{\bs x}\|_2<1\}\). Fig.~\ref{fig:gridball}
depicts different tetrahedralizations for this geometry, which are 
in particular not nested. We aim at computing the expectation 
of the solution \(u\) to the parametric diffusion equation 
\eqref{eq:modprob1} with right hand side $f\equiv 1$ and
random diffusion coefficient
\[
{{a}}({\bs y})=\bigg(\prod_{i=1}^6\frac 3 5 \big(2-y_i^2\big)\bigg)^{-1}.
\]
Since the diffusion coefficient is independent of the spatial variable, 
we can reformulate the equation according to
\[
-\Delta u({\bs y})=\prod_{i=1}^6\frac 3 5 \big(2-y_i^2\big)\text{ in }D,
\quad u({\bs y})=0\text{ on }\partial D,\quad {\bs y}\in{{\Gamma}}.
\]
Thus, since the Bochner integral interchanges with closed operators, 
see e.g.~\cite{HP57}, we obtain for the expectation of \(u\) the equation
\begin{equation}\label{eq:AnaEx}
-\Delta \E[u({\bs y})] = \E\bigg[\prod_{i=1}^6\frac 3 5 \big(2-y_i^2\big)\bigg]=1\text{ in }D,
\quad u({\bs y})=0\text{ on }\partial D,\quad {\bs y}\in{{\Gamma}}.
\end{equation}
Obviously, this equation is solved by $\E[u]({\bs x})=(1-\|{\bs x}\|_2)^2/6$.

\begin{figure}[htb]
\begin{center}
\scalebox{0.6}{
\includegraphics[width=0.35\textwidth,clip=true,trim=360 180 360 180]{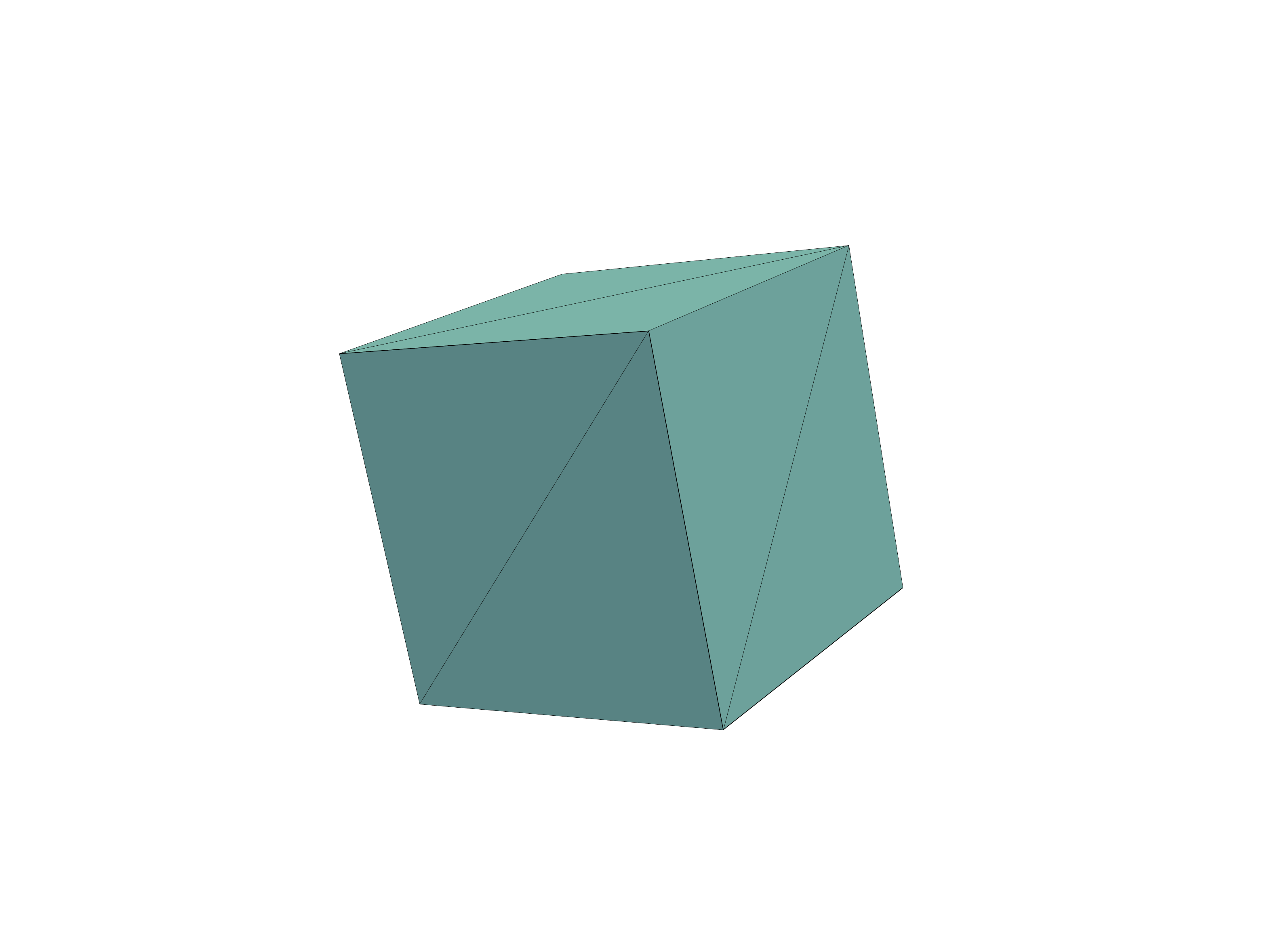}\quad
\includegraphics[width=0.35\textwidth,clip=true,trim=360 180 360 180]{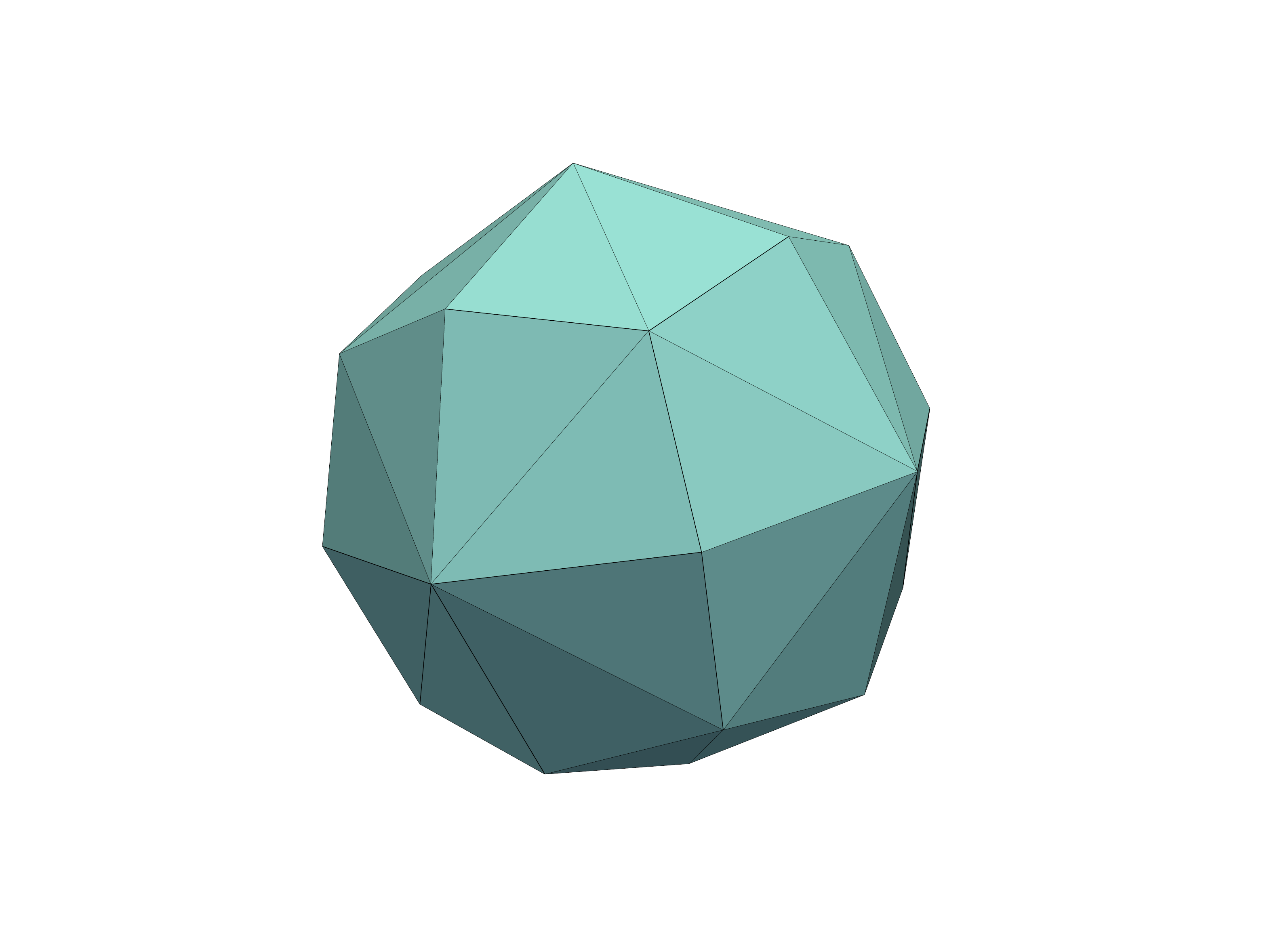}\quad
\includegraphics[width=0.35\textwidth,clip=true,trim=360 180 360 180]{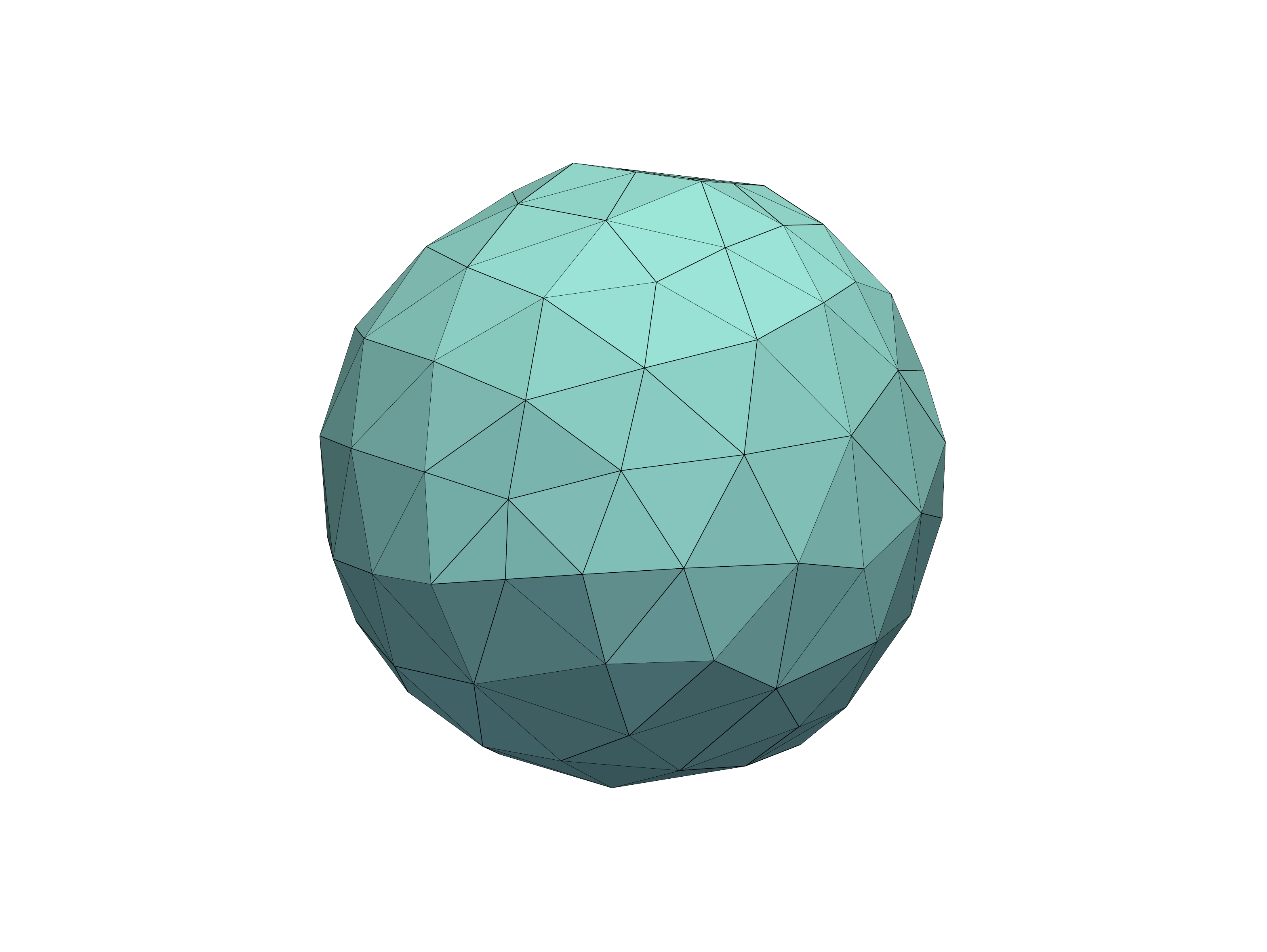}\quad
\includegraphics[width=0.35\textwidth,clip=true,trim=360 180 360 180]{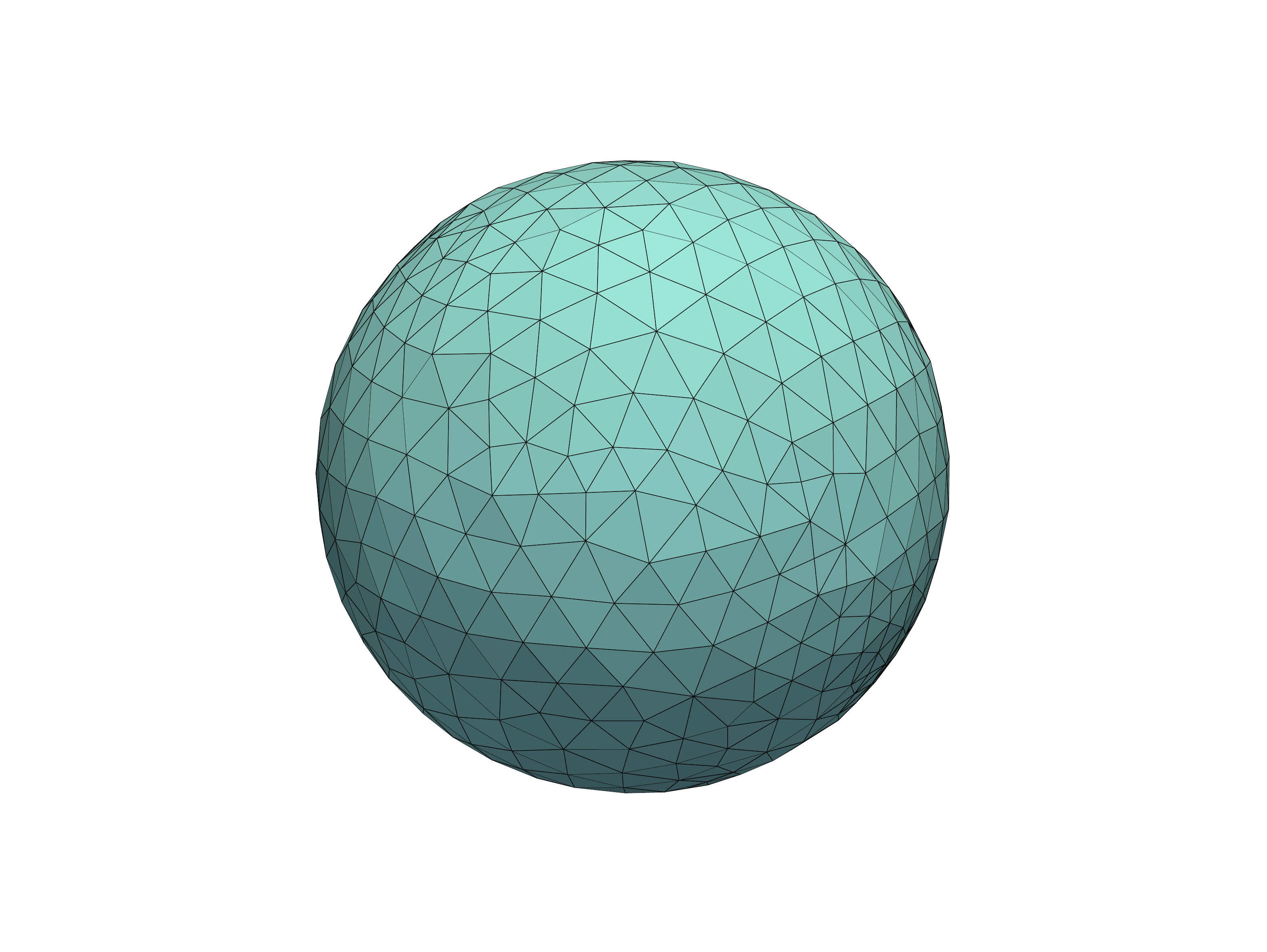}}
\caption{\label{fig:gridball}Tetrahedralizations 
of four different resolutions for the unit ball.}
\end{center}
\end{figure}

In order to measure the error to the approximate solution, 
we interpolate the exact solution to a mesh consisting of 
12\,047\,801 finite elements (this is level $j=8$). This involves 
a mesh size of \(h_8=0.0047\). For the levels \(j=0,\ldots,7\), the 
mesh sizes and corresponding degrees of freedom (DoF) are given in Table~\ref{table:meshSphere}.
Moreover, we chose \(N_0=10\) for the Monte Carlo quadrature 
and for the quasi-Monte Carlo quadrature and set 
\(N_\ell= 10\cdot 4^\ell\) and \(N_\ell=10\cdot 2^\ell\),
respectively.
For the MLMC, in order to approximate the root mean square error, 
we average five realizations of the related approximation error.
For the Clenshaw-Curtis quadrature, the number of samples are
chosen as if there would hold \(r=1\) in Lemma~\ref{lem:SGQerr}.\footnote{The Clenshaw-Curtis 
quadrature converges exponentially since the integrand is analytic.
The choice \(r=1\) is conservative and reflects the pre-asymptotic regime.}

\begin{table}[hbt]
\begin{center}
\begin{tabular}{|c||r|r|r|r|r|r|r|r|}\hline
\(\ell\) & 0 & 1 & 2 & 3 & 4 & 5 & 6 & 7\\
\hline
\(h_\ell\) & 1.2 & 0.6 & 0.3 & 0.15 & 0.075 & 0.0375 & 0.0188 & 0.0094\\ \hline
\(\operatorname{dof}_\ell\) & 8    &      27   &      244      &  1585    &    6042  &     29069   &   133376  &    551327    \\
\hline
\end{tabular}
\caption{\label{table:meshSphere}Mesh sizes and DoF on the different levels for the unit ball.}
\end{center}
\end{table}

\begin{figure}[htb]
\begin{center}
\includegraphics[width=0.45\textwidth,clip=true,trim=0 0 20 20]{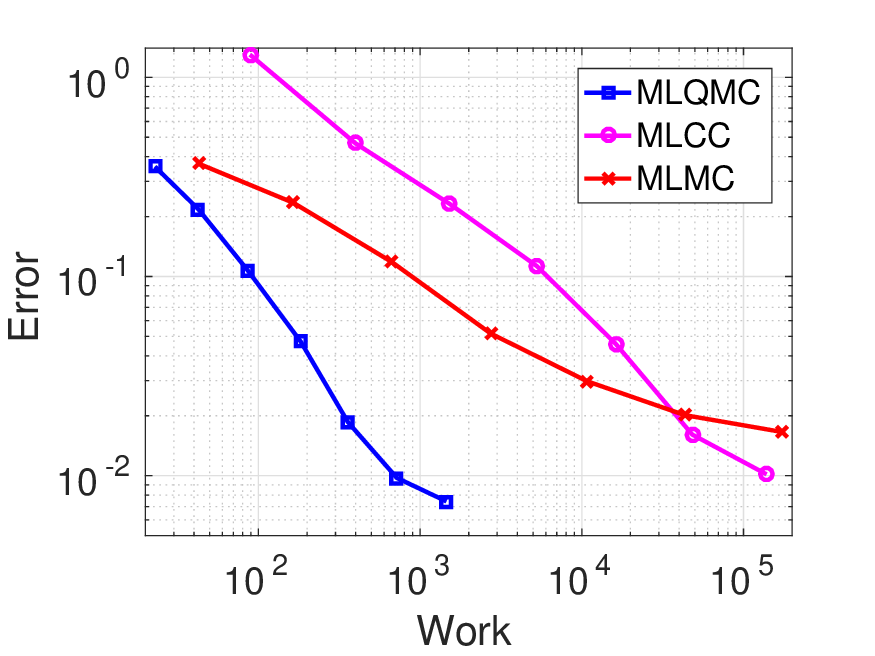}\qquad
\includegraphics[width=0.45\textwidth,clip=true,trim=0 0 40 20]{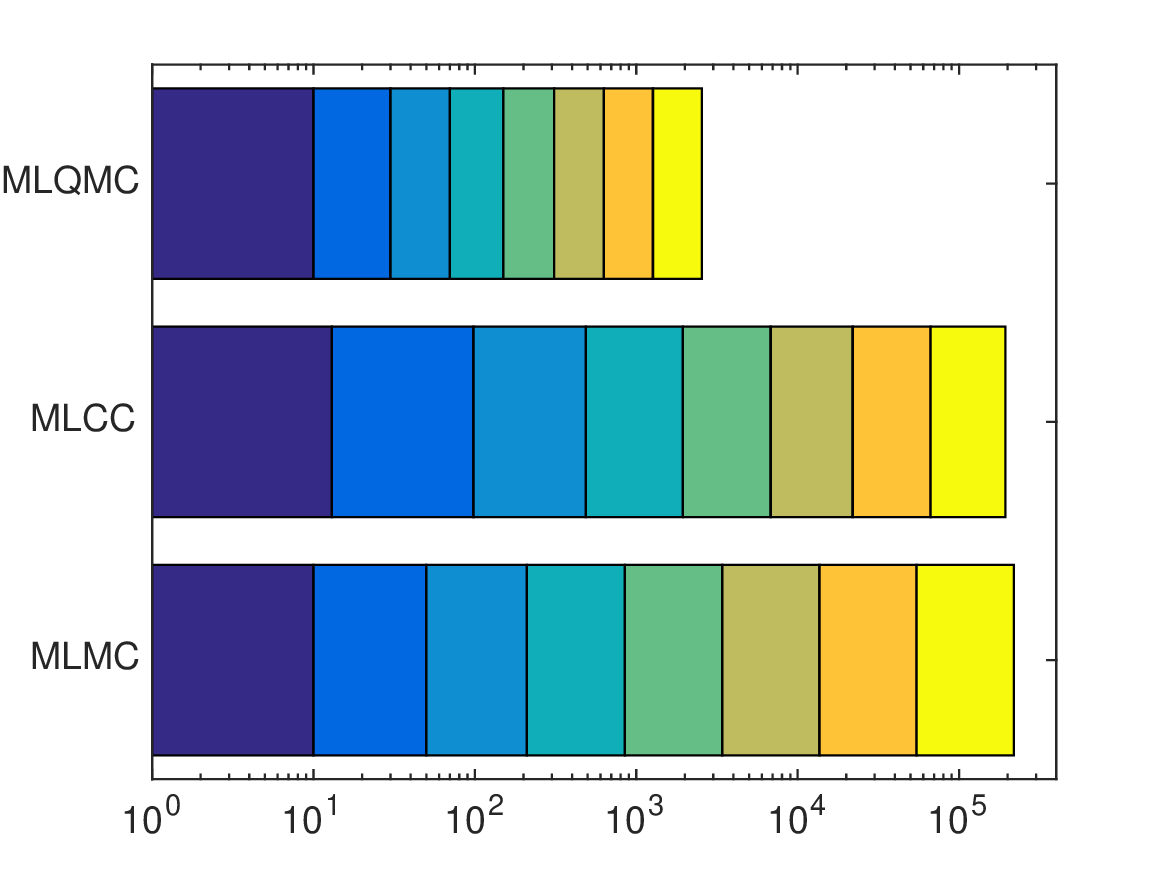}
\caption{\label{fig:ResSphere}\(H^1\)-errors of the different quadrature 
methods (left) and number of samples on each level in case of \(j=7\) (right) for the unit ball.}
\end{center}
\end{figure}

On the left side of Fig.~\ref{fig:ResSphere}, the error for the 
MLQMC, the MLCC and the MLMC is visualized. It is plotted 
against the work, which is expressed in terms of fine grid
samples: In accordance with the degrees of freedom denoted 
in Table~\ref{table:meshSphere}, we scale each sample on a 
particular level $\ell$ with the factor \(\operatorname{DoF}_{\ell}
/\operatorname{DoF}_j\), i.e.\ we weight a fine grid sample by 
$1$ and scale the coarse grid samples accordingly. The work 
is then given by summing up the total number of samples per 
level times the related weight.

It can be seen that MLQMC achieves the best error 
versus work rate. Moreover, the plot indicates that MLCC 
may asymptotically achieve a similar rate. MLMC seems
to provide here only a halved rate compared to MLQMC.
To give an insight on the number of samples spent on each 
particular level, we have depicted the corresponding numbers 
for \(j=7\) on the right hand side of Fig.~\ref{fig:ResSphere}.
It turns out that the quasi-Monte Carlo quadrature requires the 
smallest number of quadrature points. In contrast, the number 
of points for the Monte Carlo quadrature and for the Clenshaw-Curtis 
quadrature are nearly the same. This may be caused by the 
conservative choice for the number of quadrature points for 
the latter. Nevertheless, for fixed parameter dimension \(m\) 
and \(r=1\), we expect asymptotically similar rates for MLCC and MLQMC.

\subsection{A more complex example}
In our second example, the spatial domain is given by a 
model of the Zarya module of the International Space Station 
(ISS), which was the first module to be launched.\footnote{We 
thank Martin Siegel (Rheinbach, Germany) 
who kindly provided us with this model.}
Fig.~\ref{fig:grids} shows different tetrahedralizations of 
this geometry with decreasing mesh size. Note that the 
geometry can be imbedded into a cylinder with radius 
$0.52$ and height $1.58$. 

\begin{figure}[htb]
\begin{center}
\scalebox{0.6}{
\includegraphics[width=0.35\textwidth,clip=true,trim=260 125 230 150]{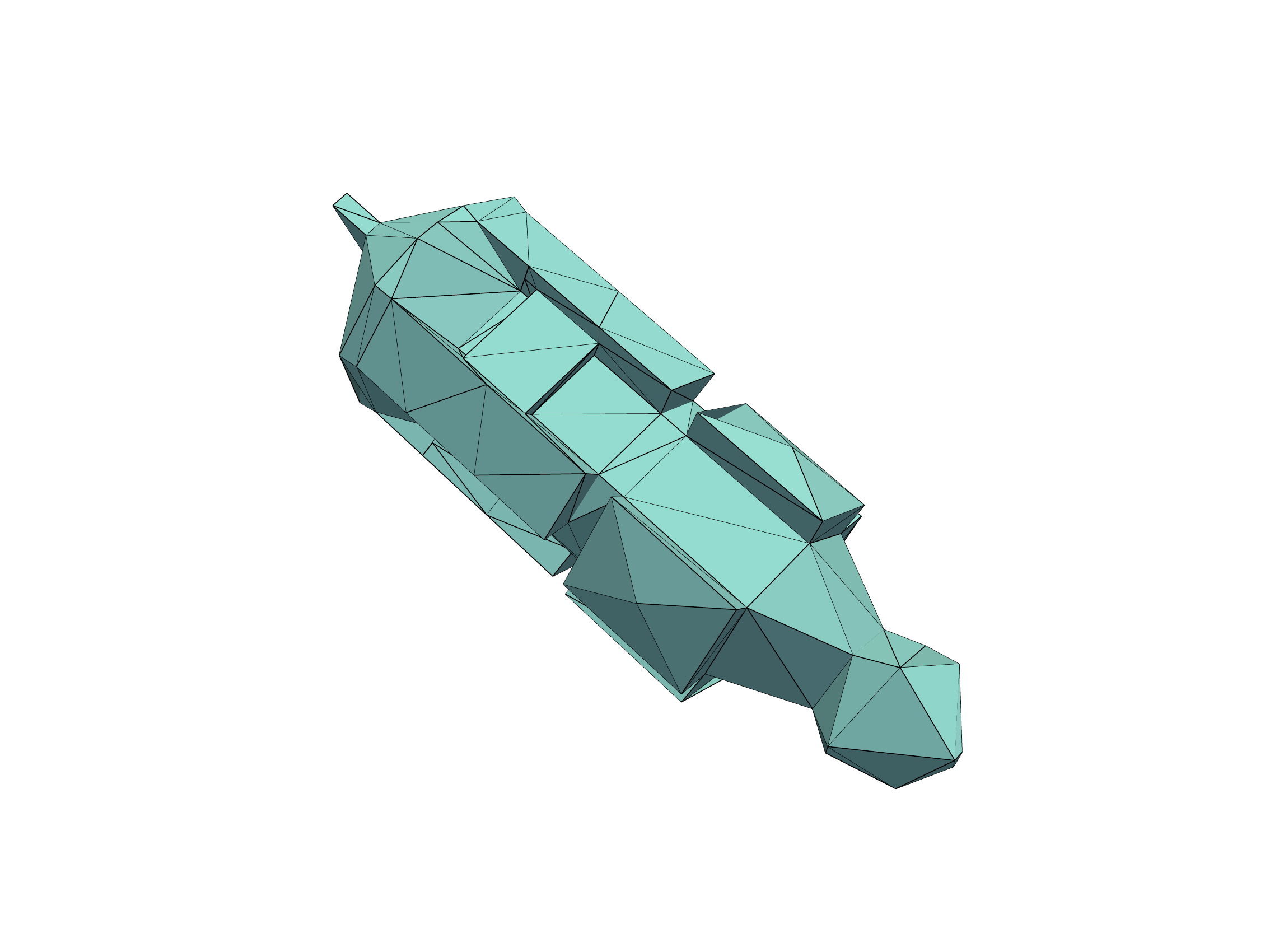}\quad
\includegraphics[width=0.35\textwidth,clip=true,trim=260 125 230 150]{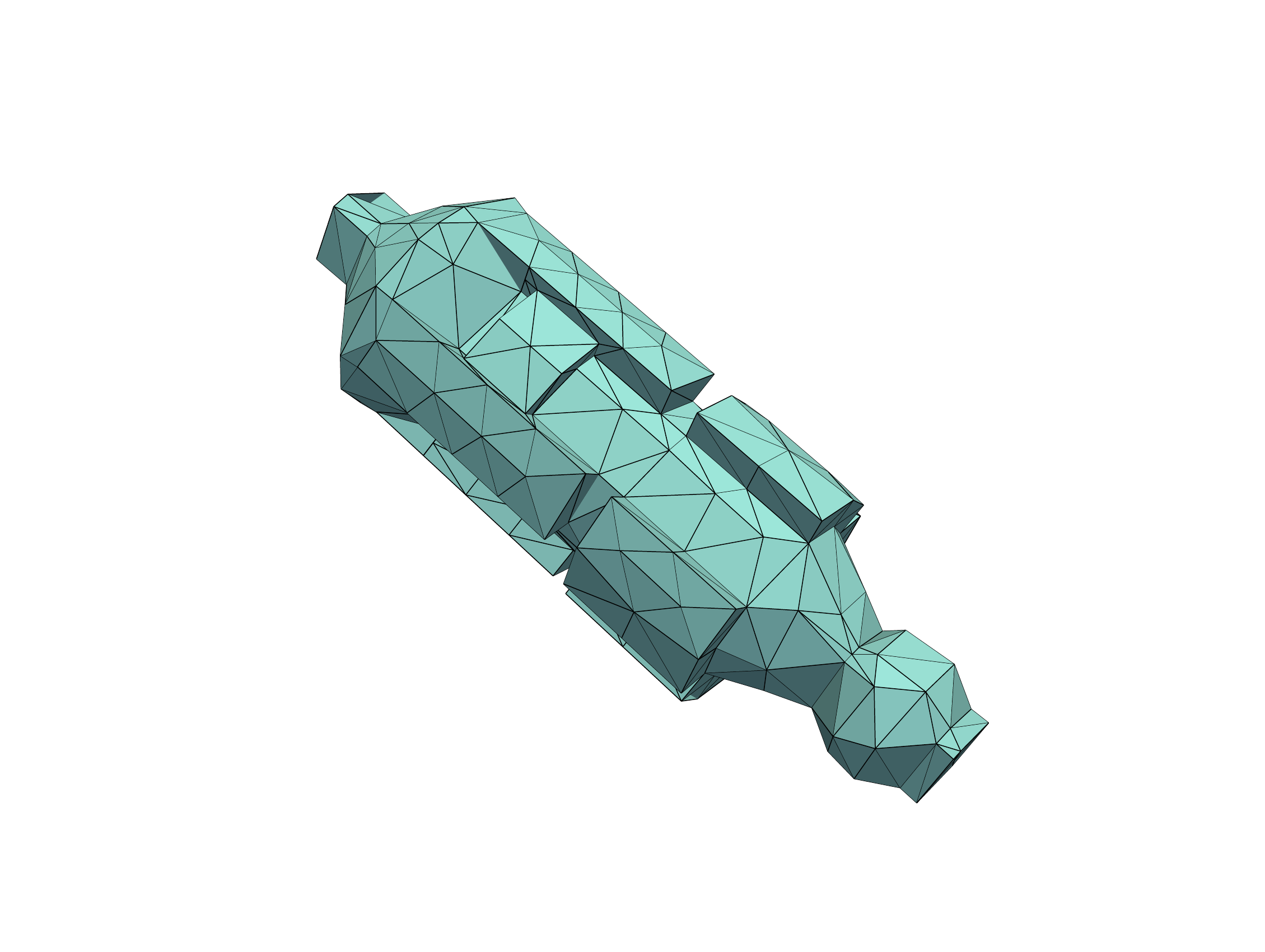}\quad
\includegraphics[width=0.35\textwidth,clip=true,trim=260 125 230 150]{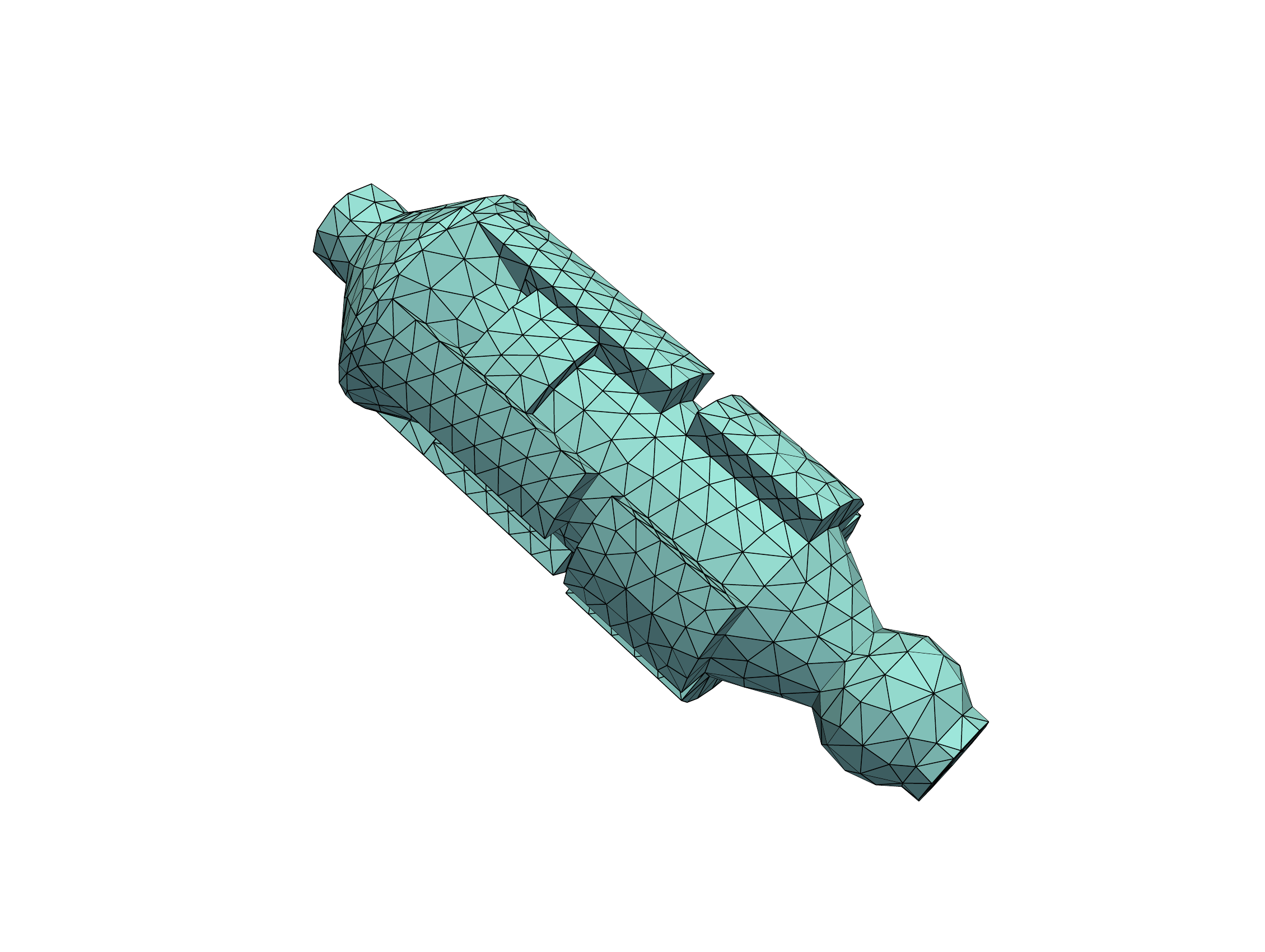}\quad
\includegraphics[width=0.35\textwidth,clip=true,trim=260 125 230 150]{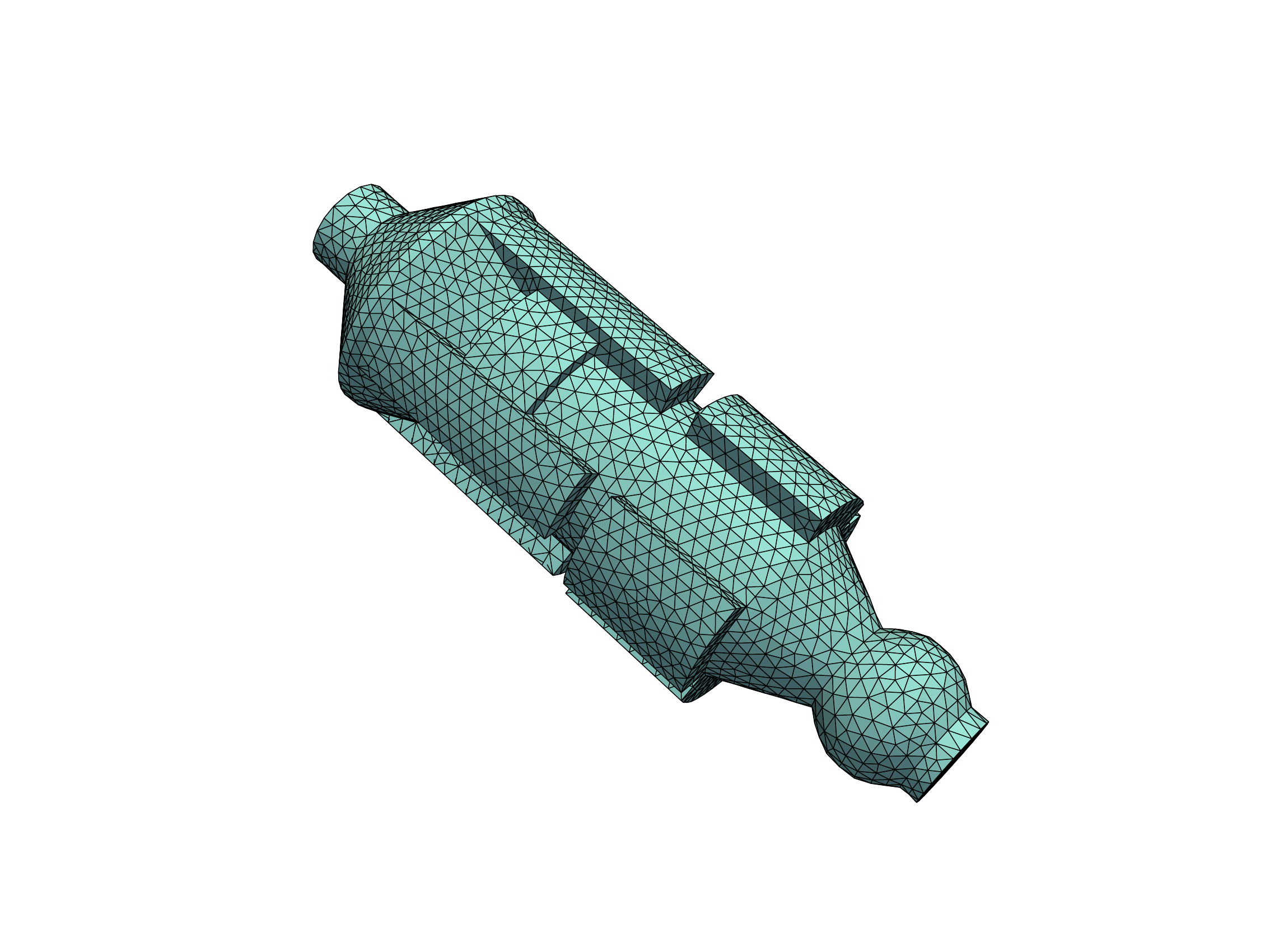}}
\caption{\label{fig:grids}Tetrahedralizations of four
different resolutions for the Zarya geometry.}
\end{center}
\end{figure}

\begin{figure}[htb]
\begin{center}
\includegraphics[width=0.35\textwidth,clip=true,trim=1000 300 900 400]{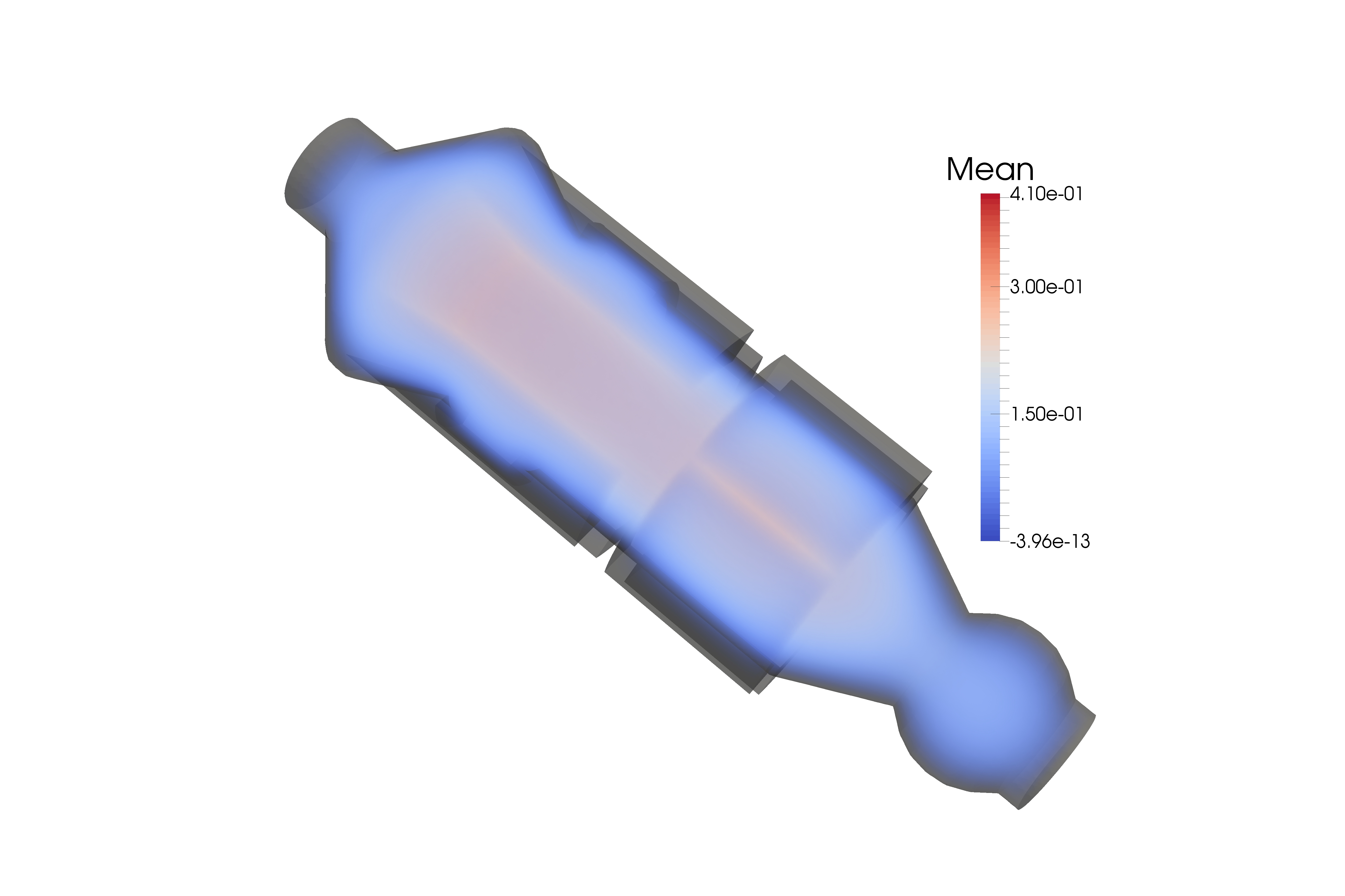}\qquad
\includegraphics[width=0.35\textwidth,clip=true,trim=1000 300 900 400]{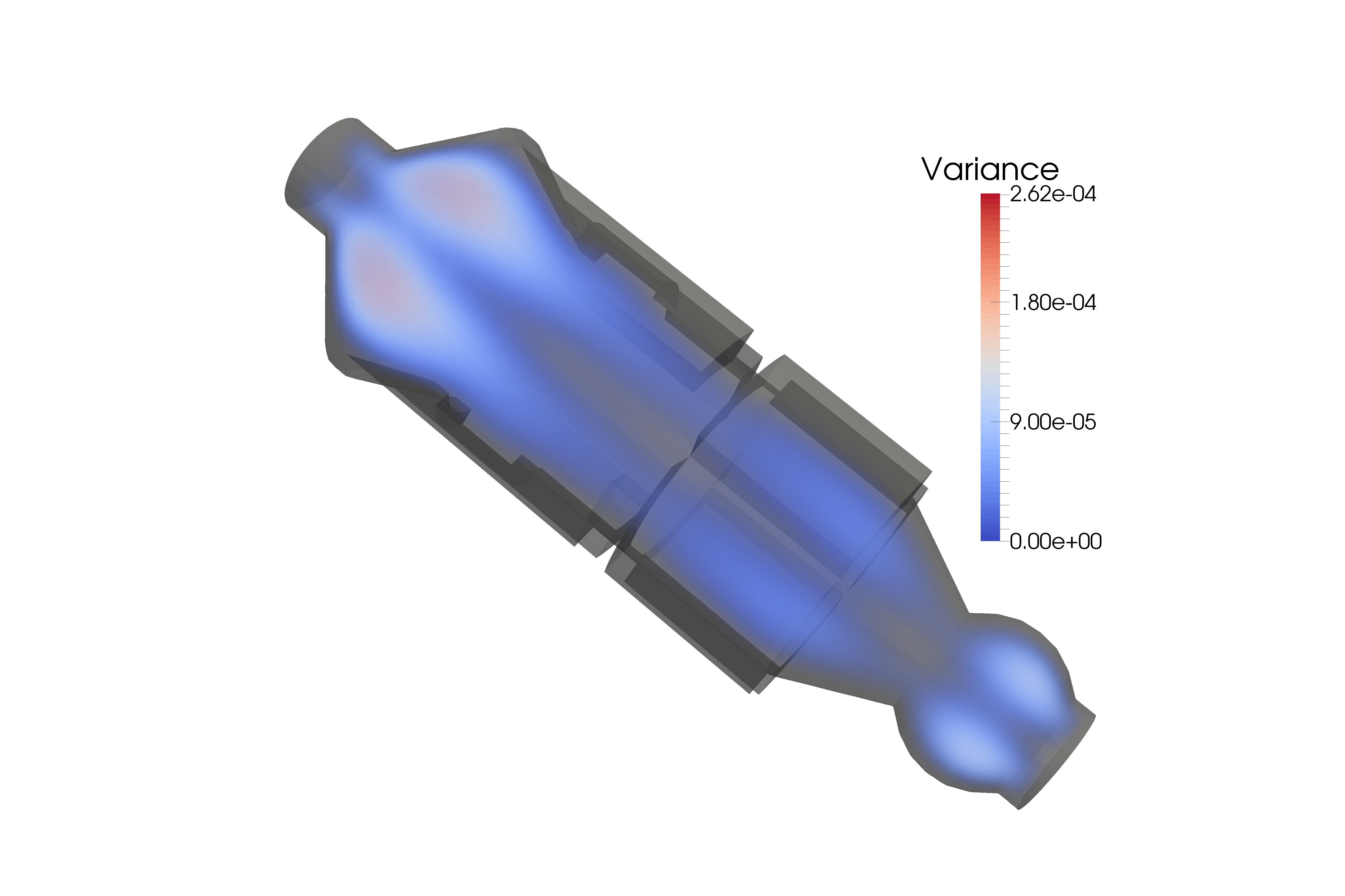}
\caption{\label{fig:ISS}Mean (left) and variance (right) of the model problem on
the Zarya geometry.}
\end{center}
\end{figure}

\begin{table}[hbt]
\begin{center}
\begin{tabular}{|c||r|r|r|r|r|r|r|}\hline
\(j\) & 0 & 1 & 2 & 3 & 4 & 5 & 6\\
\hline
\(h_j\) & 0.5 & 0.25 & 0.125 & 0.0625 & 0.0313 & 0.0156 & 0.0078\\ \hline
\(\operatorname{dof}_j\) & 174 & 333 & 1240 & 5846 & 30171 & 141029 & 617111 \\
\hline
\end{tabular}
\caption{\label{table:meshISS}Mesh sizes and DoF on the different levels for the Zarya geometry.}
\end{center}
\end{table}

In this example, the parametric diffusion coefficient 
is given by
\begin{align*}
{{a}}({\bs x},{\bs y})=1&+\frac{\exp({\|{\bs x}\|_2^2})}{20}\bigg(\sin(2\pi x_1)y_1+\frac 1 2 \sin(2\pi x_2)y_2+\frac 1 4 \sin(2\pi x_3)y_3\\
&+\frac 1 8 \sin(4\pi x_1)\sin(4\pi x_2)y_4+\frac{1}{16}\sin(4\pi x_1)\sin(4\pi x_3)y_5\\
&+\frac{1}{32}\sin(4\pi x_2)\sin(4\pi x_3)y_6\bigg)
\end{align*}
and \(f=10\). For \({\bs x}\in D\) and \({\bs y}\in{{\Gamma}}\), the 
diffusion coefficient varies approximately in the range \([0.19,1.81]\). 
Fig.\ \ref{fig:ISS} shows the mean (left) and the variance (right) of the 
reference solution. It has been computed on a mesh with 13\,069\,396 tetrahedrons
resulting in a mesh size of \(h=0.0039\) by 10\,000 quasi-Monte Carlo 
samples based on the Halton sequence.
For the levels \(j=0,\ldots,6\), the mesh sizes and corresponding DoF are given in Table~\ref{table:meshISS}.

Fig.~\ref{fig:ISSerr} visualizes the errors of the approximate
expectation and second moment for the different multilevel quadrature 
methods under consideration. 
The number of quadrature points 
for the presented methods are chosen as in the previous example. 
Again, MLQMC provides the best error versus work rate in the mean, as well as
in the second moment. The rates of MLMC and MLCC are both lower here.

\begin{figure}[htb]
\begin{center}
\includegraphics[width=0.46\textwidth,clip=true,trim= 0 0 20 20]{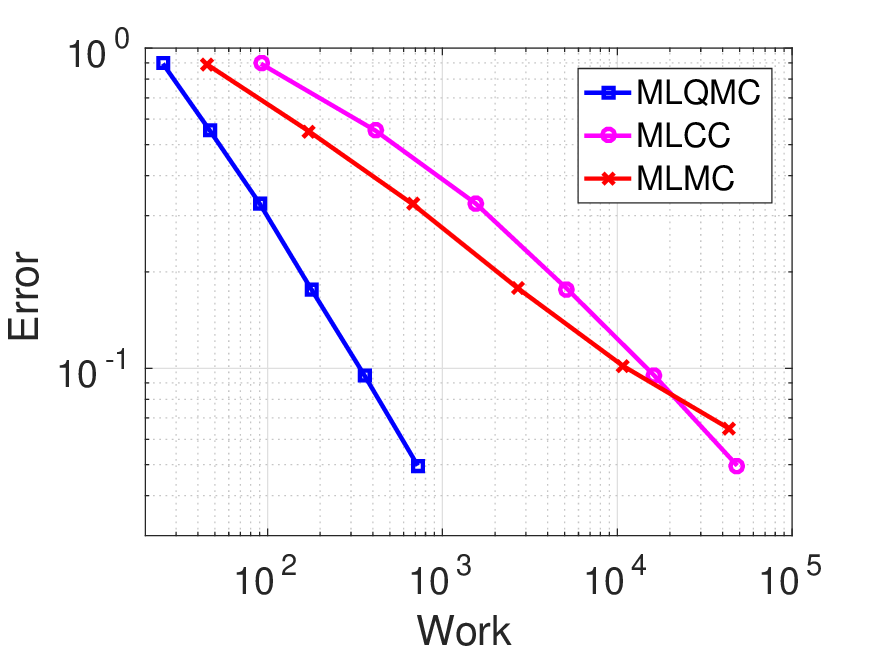}\qquad
\includegraphics[width=0.46\textwidth,clip=true,trim=0 0 20 20]{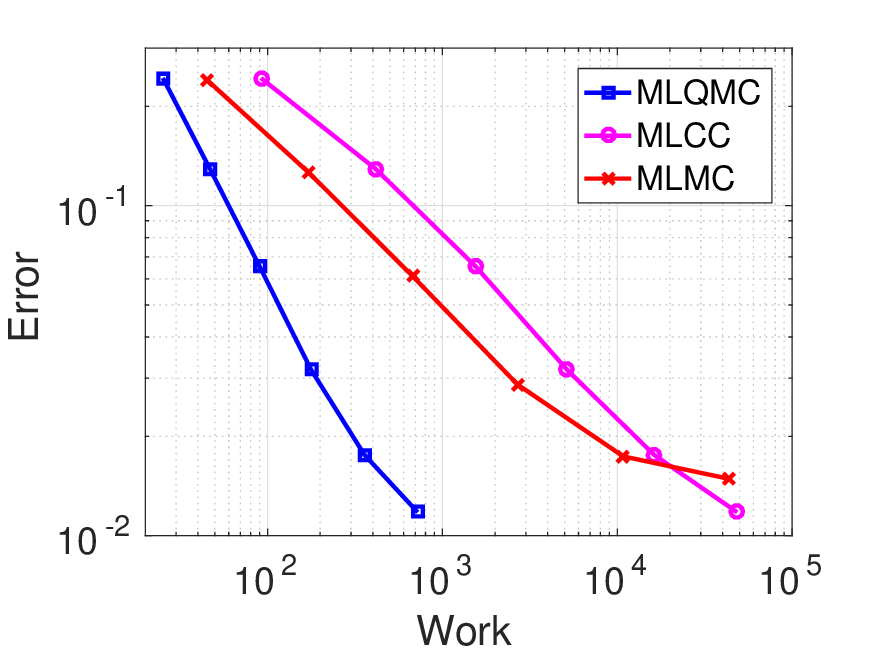}
\caption{\label{fig:ISSerr}\(H^1\)-errors of the approximate mean (left) and \(W^{1,1}\)-errors 
of the approximate second moment (right) on the Zarya geometry for different quadrature methods.}
\end{center}
\end{figure}

\section{Conclusion}
In the present article, we have reversed the construction of the
conventional multilevel quadrature. This enables us to give up the 
nestedness of the spatial approximation spaces. 
{In particular, a polygonal approximation of curved domain boundaries 
is sufficient for computing the finite element solution.} Hence, black-box 
finite element solvers can be directly applied to compute the solution of 
the underlying boundary value problem. 
Note that adaptively refined finite 
element meshes can be easily used as well.
Another aspect of our approach 
is that the cost is considerably reduced by the application of nested 
quadrature formulae. Both features have been demonstrated 
by numerical results for the Clenshaw-Curtis quadrature and 
the quasi-Monte Carlo quadrature based on Halton points. Of 
course, other nested quadrature formulae like the Gauss-Patterson 
quadrature can be used as well. The application of quadrature 
formulae which are tailored to a possible anisotropy of the integrand 
is also straightforward. If non-nested quadrature formulae are applied, 
one arrives at a combination-technique-like representation of the 
multilevel quadrature.

\bibliographystyle{plain}
\bibliography{bibl}

\begin{thebibliography}{10}

\bibitem{BNT}
I.~Babu{\v{s}}ka, F.~Nobile, and R.~Tempone.
\newblock A stochastic collocation method for elliptic partial differential
  equations with random input data.
\newblock {\em SIAM J. Numer. Anal.}, 45(3):1005--1034, 2007.

\bibitem{BTZ}
I.~Babu{\v{s}}ka, R.~Tempone, and G.~Zouraris.
\newblock Galerkin finite element approximations of stochastic elliptic partial
  differential equations.
\newblock {\em SIAM J. Numer. Anal.}, 42(2):800--825, 2004.

\bibitem{BSZ}
A.~Barth, C.~Schwab, and N.~Zollinger.
\newblock Multi-level {M}onte {C}arlo finite element method for elliptic {PDEs}
  with stochastic coefficients.
\newblock {\em Numer. Math.}, 119(1):123--161, 2011.

\bibitem{BNTT12}
J.~Beck, R.~Tempone, F.~Nobile, and L.~Tamellini.
\newblock On the optimal polynomial approximation of stochastic {PDE}s by
  {G}alerkin and collocation methods.
\newblock {\em Math. Models Methods Appl. Sci.}, 22(09):1250023, 2012.

\bibitem{B}
D.~Braess.
\newblock {\em Finite {E}lements. Theory, Fast Solvers, and Applications in
  Solid Mechanics}.
\newblock Cambridge University Press, Cambridge, 2nd edition, 2001.

\bibitem{Brenner}
S.~Brenner and L.~Scott.
\newblock {\em The {M}athematical {T}heory of {F}inite {E}lement {M}ethods}.
\newblock Springer, Berlin, 3rd edition, 2008.

\bibitem{BG}
H.-J. Bungartz and M.~Griebel.
\newblock Sparse grids.
\newblock {\em Acta Numer.}, 13:147--269, 2004.

\bibitem{CST13}
J.~Charrier, R.~Scheichl, and A.~L. Teckentrup.
\newblock Finite element error analysis of elliptic {PDE}s with random
  coefficients and its application to multilevel monte carlo methods.
\newblock {\em SIAM J. Numer. Anal.}, 51(1):322--352, 2013.

\bibitem{CDS}
A.~Cohen, R.~DeVore, and C.~Schwab.
\newblock Convergence rates of best {$N$}-term {G}alerkin approximations for a
  class of elliptic s{PDE}s.
\newblock {\em Found. Comput. Math.}, 10:615--646, 2010.

\bibitem{CDS11}
A.~Cohen, R.~DeVore, and C.~Schwab.
\newblock Analytic regularity and polynomial approximation of parametric and
  stochastic elliptic {PDE}s.
\newblock {\em Anal. Appl.}, 09(01):11--47, 2011.

\bibitem{ES14}
O.~Ernst and B.~Sprungk.
\newblock Stochastic collocation for elliptic {PDE}s with random data: The
  lognormal case.
\newblock In J.~Garcke and D.~Pfl{\"u}ger, editors, {\em Sparse Grids and
  Applications --- Munich 2012}, pages 29--53. Springer International
  Publishing, Cham, 2014.

\bibitem{FST}
P.~Frauenfelder, C.~Schwab, and R.~Todor.
\newblock Finite elements for elliptic problems with stochastic coefficients.
\newblock {\em Comput. Methods Appl. Mech. Engrg.}, 194(2-5):205--228, 2005.

\bibitem{GG}
T.~Gerstner and M.~Griebel.
\newblock Numerical integration using sparse grids.
\newblock {\em Numer. Algorithms}, 18:209--232, 1998.

\bibitem{GH}
T.~Gerstner and S.~Heinz.
\newblock Dimension- and time-adaptive multilevel {M}onte {C}arlo methods.
\newblock In J.~Garcke and M.~Griebel, editors, {\em Sparse {G}rids and
  {A}pplications}, volume~88 of {\em Lecture Notes in Computational Science and
  Engineering}, pages 107--120, Berlin-Heidelberg, 2012. Springer.

\bibitem{GS}
R.~Ghanem and P.~Spanos.
\newblock {\em Stochastic Finite Elements. {A} {S}pectral {A}pproach}.
\newblock Springer, New York, 1991.

\bibitem{G}
M.~Giles.
\newblock Multilevel {M}onte {C}arlo path simulation.
\newblock {\em Oper. Res.}, 56(3):607--617, 2008.

\bibitem{Gil15}
M.~Giles.
\newblock Multilevel {M}onte {C}arlo methods.
\newblock {\em Acta Numer.}, 24:259--328, 2015.

\bibitem{GW}
M.~Giles and B.~Waterhouse.
\newblock Multilevel quasi-{M}onte {C}arlo path simulation.
\newblock {\em Radon Series Comp. Appl. Math.}, 8:1--18, 2009.

\bibitem{GH1}
M.~Griebel and H.~Harbrecht.
\newblock On the construction of sparse tensor product spaces.
\newblock {\em Math. Comput.}, 82(282):975--994, 2013.

\bibitem{HANvST15}
A.-L. Haji-Ali, F.~Nobile, E.~von Schwerin, and R.~Tempone.
\newblock Optimization of mesh hierarchies in multilevel {M}onte {C}arlo
  samplers.
\newblock {\em Stoch. Partial Differ. Equ. Anal. Comput.}, 4(1):76--112, 2016.

\bibitem{Hal60}
J.~Halton.
\newblock On the efficiency of certain quasi-random sequences of points in
  evaluating multi-dimensional integrals.
\newblock {\em Numer. Math.}, 2(1):84--90, 1960.

\bibitem{HPS1}
H.~Harbrecht, M.~Peters, and M.~Siebenmorgen.
\newblock On multilevel quadrature for elliptic stochastic partial differential
  equations.
\newblock In J.~Garcke and M.~Griebel, editors, {\em Sparse Grids and
  Applications}, volume~88 of {\em Lecture Notes in Computational Science and
  Engineering}, pages 161--179, Berlin-Heidelberg, 2012. Springer.

\bibitem{HPS14a}
H.~Harbrecht, M.~Peters, and M.~Siebenmorgen.
\newblock Efficient approximation of random fields for numerical applications.
\newblock {\em Numer. Linear Algebra Appl.}, 22(4):596--617, 2015.

\bibitem{HPS14b}
H.~Harbrecht, M.~Peters, and M.~Siebenmorgen.
\newblock Analysis of the domain mapping method for elliptic diffusion problems
  on random domains.
\newblock {\em Numer. Math.}, 134(4):823--856, 2016.

\bibitem{HPS2}
H.~Harbrecht, M.~Peters, and M.~Siebenmorgen.
\newblock Multilevel accelerated quadrature for {PDE}s with log-normally
  distributed diffusion coefficient.
\newblock {\em SIAM/ASA J. Uncertain. Quantif.}, 4(1):520--551, 2016.

\bibitem{HPS16}
H.~Harbrecht, M.~Peters, and M.~Siebenmorgen.
\newblock On the quasi-{M}onte {C}arlo method with {H}alton points for elliptic
  {PDE}s with log-normal diffusion.
\newblock {\em Math. Comp.}, 86:771--797, 2017.

\bibitem{H}
S.~Heinrich.
\newblock The multilevel method of dependent tests.
\newblock In {\em Advances in stochastic simulation methods ({S}t.
  {P}etersburg, 1998)}, Stat. Ind. Technol., pages 47--61. Birkh\"auser,
  Boston, MA, 2000.

\bibitem{H2}
S.~Heinrich.
\newblock Multilevel {M}onte {C}arlo methods.
\newblock In {\em Lecture Notes in Large Scale Scientific Computing}, pages
  58--67, London, 2001. Springer.

\bibitem{HP57}
E.~Hille and R.~Phillips.
\newblock {\em Functional {A}nalysis and {S}emi-{G}roups}, volume~31 of {\em
  American Mathematical Society Colloquium Publications}.
\newblock American Mathematical Society, Providence, 1957.

\bibitem{HS11}
V.~Hoang and C.~Schwab.
\newblock {$N$}-term {W}iener chaos approximation rate for elliptic {PDE}s with
  lognormal {G}aussian random inputs.
\newblock {\em Math. Models Methods Appl. Sci.}, 4(24):797–826, 2014.

\bibitem{KSS3}
F.~{Kuo}, C.~{Schwab}, and I.~{Sloan}.
\newblock {Multi-level quasi-Monte Carlo finite element methods for a class of
  elliptic partial differential equations with random coefficients}.
\newblock {\em Found. Comput. Math.}, 15(2):411--449, 2015.

\bibitem{L}
M.~Lo{\`e}ve.
\newblock {\em Probability theory. {I+I\!I}}, volume~45 of {\em Graduate Texts
  in Mathematics}.
\newblock Springer, New York, 4th edition, 1977.

\bibitem{MK}
H.~Matthies and A.~Keese.
\newblock Galerkin methods for linear and nonlinear elliptic stochastic partial
  differential equations.
\newblock {\em Comput. Methods Appl. Mech. Engrg.}, 194(12-16):1295--1331,
  2005.

\bibitem{Niederreiter}
H.~Niederreiter.
\newblock {\em Random {N}umber {G}eneration and {Q}uasi-{M}onte {C}arlo
  Methods}.
\newblock Society for Industrial and Applied Mathematics, Philadelphia, PA,
  USA, 1992.

\bibitem{NR96}
E.~Novak and K.~Ritter.
\newblock High dimensional integration of smooth functions over cubes.
\newblock {\em Numer. Math.}, 75(1):79--97, 1996.

\bibitem{ST2}
C.~Schwab and R.~Todor.
\newblock Karhunen-{L}o\`eve approximation of random fields by generalized fast
  multipole methods.
\newblock {\em J. Comput. Phys.}, 217:100--122, 2006.

\bibitem{Sieb15}
M.~Siebenmorgen.
\newblock {\em Quadrature methods for elliptic {PDE}s with random diffusion}.
\newblock PhD Thesis, Faculty of Science, University of Basel, 2015.

\bibitem{TSU13}
A.~Teckentrup, R.~Scheichl, M.~Giles, and E.~Ullmann.
\newblock Further analysis of multilevel {M}onte {C}arlo methods for elliptic
  {PDE}s with random coefficients.
\newblock {\em Numer. Math.}, 125(3):569--600, 2013.

\bibitem{TJWG15}
A.~L. Teckentrup, P.~Jantsch, C.~G. Webster, and M.~Gunzburger.
\newblock A multilevel stochastic collocation method for partial differential
  equations with random input data.
\newblock {\em SIAM/ASA J. Uncertain. Quantif.}, 3(1):1046--1074, 2015.

\bibitem{ST3}
R.~Todor and C.~Schwab.
\newblock Convergence rates for sparse chaos approximations of elliptic
  problems with stochastic coefficients.
\newblock {\em IMA J. Numer. Anal.}, 27(2):232--261, 2007.

\bibitem{Wan02}
X.~Wang.
\newblock A constructive approach to strong tractability using quasi-{M}onte
  {C}arlo algorithms.
\newblock {\em J. Complexity}, 18:683--701, 2002.

\end{thebibliography}
\end{document}